\documentclass[reqno]{amsart}
\newcommand \datum {November 7, 2023}
\usepackage{amssymb,latexsym}
\usepackage{amsmath}
\usepackage{hyperref}
\usepackage{url}
\usepackage{graphicx}
\usepackage[dvipsnames]{xcolor}
\usepackage{enumerate}
\usepackage{tikz, color}
\numberwithin{equation}{section}
\theoremstyle{plain}
 \newtheorem{theorem}{Theorem}[section]
  
 \newtheorem{lemma}[theorem]{Lemma}
 \newtheorem{proposition}[theorem]{Proposition}
 
 \newtheorem{observation}[theorem]{Observation} 
 
 \newtheorem{corollary}[theorem]{Corollary}
\theoremstyle{definition}
 \newtheorem{definition}[theorem]{Definition}
 
 \newtheorem{example}[theorem]{Example}
 \newtheorem{conjecture}[theorem]{Conjecture} \newtheorem{remark}[theorem]{Remark}
\theoremstyle{remark}

\newcommand  \FSP[3]{\textup{FSP}(#1,#2,#3)}
\newcommand  \fprab[4]{f_{#2,#3,#4}^{(#1)}}
\newcommand  \frab[3]{f_{#1,#2,#3}}
\newcommand  \flatrab[3]{f_{#1,#2,#3}^\flat}
\newcommand \FD[1] {\textup{FD}(#1)}
\newcommand \crn {W_3}
\newcommand \icrn[1]{W_3^{(#1)}}
\newcommand \inseg[2] {\textup{Is}(#1,#2)}
\newcommand \pset[1] {\textup{Ps}(#1)}
\newcommand \bul[1] {#1^\bullet}
\newcommand \auxf {\fhb}
\newcommand\gap{13} 
\newcommand\stp{1.5}
\newcommand\bstp{2}
\newcommand\hos{3.5}
\newcommand\lntxt{1.5}
\newcommand\rntxt{1.5}
\newcommand \bcirc[1] {\fill[black] (#1) circle (5pt);
  \draw [black,thick] (#1)  circle [radius=5pt]}
\newcommand \wcirc[1] {\fill[white] (#1) circle (5pt);
  \draw [black,thick] (#1)  circle [radius=5pt]}
\newcommand \nodetextl[2] {\draw [white,fill] (#1) circle [radius=0.0pt] node [black,left=\lntxt] {#2} } 
\newcommand \nodetextr[2] {\draw [white,fill] (#1) circle [radius=0.0pt] node [black,right=\rntxt] {#2} } 
\newcommand \nodetextb[2] {\draw [white,fill] (#1) circle [radius=0.0pt] node [black,below=\rntxt] {#2} } 
\newcommand \nodetexta[2] {\draw [white,fill] (#1) circle [radius=0.0pt] node [black,above=\rntxt] {#2} } 

\newcommand \nodeutxt[2] {\draw [white,fill] (#1) circle [radius=0.0pt] node [black,above=-1pt] {#2} } 
\newcommand \nodeuutxt[2] {\draw [white,fill] (#1) circle [radius=0.0pt] node [black,above=1pt] {#2} }

\newcommand \nodeelltxt[2] {\draw [white,fill] (#1) circle [radius=0.0pt] node [black,above=-6.5pt] {\tbf{#2}}} 
\newcommand \nodewhtxt[2] {\draw [white,fill] (#1) circle [radius=0.0pt] node [white,above=-6.5pt] {\tbf{#2}}}

\newcommand\hot{30pt}
\newcommand\vot{20pt}
\newcommand\het{\vot}
\newcommand\wit{50pt}
\newcommand\whh{94pt} % {80pt}

\newcommand \acr[1] {\raisebox{1pt}{\ul{#1}}}
\newcommand \ul[1] {\underline{#1}}
\newcommand \ibinom[2] {\textup C_{\textup{b}}(#1,#2)}
\newcommand \Pow [1]{\textup{Pow}(#1)}
\newcommand \nPow {\Pow{[n]}}
\newcommand \ntPow [1] {\textup{Pow}_{\textup{nt}}(#1)}
\newcommand\Sp {\textup{Sp}}
\newcommand\fsp {f_{\textup{Sp}}}
\newcommand\gmin [1]{\textup{Gm}(#1)}
\newcommand \Jir [1]{\textup{J}(#1)}
\newcommand \vs{\vec\sigma}
\newcommand\lint[1]{\lfloor #1\rfloor}
\newcommand\uint[1]{\lceil #1\rceil}
\renewcommand \phi{\varphi}
\newcommand \Nnul {{\mathbb N}_0}
\newcommand \Nplu {{\mathbb N^+}}
\newcommand \Nfromthree {{\mathbb N^{\geq 3}}} % ???
\newcommand \Nfromwhat[1]{{\mathbb N^{\geq #1}}}
\newcommand \sct[1]{\scriptstyle{#1}}

\newcommand \faux {f_{\textup{aux}}}
\newcommand \fgr[1] {g_{#1}}
\newcommand \fgha {g_{3}^{\ast}}
\newcommand \fghb {g_{3}^{\ast\ast}}
\newcommand \fha{f_{3,4}}
\newcommand \fhb{f_{3,3}}

\newcommand{\tbf}{\textbf}% text bold
\newcommand{\set}[1]{\{#1\}}% set 
\newcommand \Sym[1] {\textup{Sym}(#1)}

\newcommand \red[1]{{\textcolor{red}{#1}\color{black}}}

\begin{document}

\title[Generating the powers of  free distributive lattices]
{Minimum-sized generating sets of the direct powers of  free distributive lattices}

\author[G.\ Cz\'edli]{G\'abor Cz\'edli}
\email{czedli@math.u-szeged.hu}
\urladdr{http://www.math.u-szeged.hu/~czedli/}
\address{University of Szeged, Bolyai Institute. 
Szeged, Aradi v\'ertan\'uk tere 1, HUNGARY 6720}

\begin{abstract} 
For a finite lattice $L$, let Gm($L$) denote the least $n$ such that $L$ can be generated by $n$ elements. For integers $r>2$ and $k>1$,  denote by FD$(r)^k$   the $k$-th direct power of the free distributive lattice FD($r$) on $r$ generators.  We determine Gm(FD$(r)^k$)  for many pairs $(r,k)$ either exactly or with good accuracy by giving a lower estimate that becomes an upper estimate if we increase it by 1. 
% for it such that the difference between these two estimates is only 1. 
For example, for $(r,k)=(5,25\,000)$ and  $(r,k)=(20,\ 1.489\cdot 10^{1789})$, Gm(FD$(r)^k$) is 300 and 6000, respectively.  To reach our goal, we give estimates for the maximum number of pairwise unrelated copies of some specific posets (called full segment posets) in the subset lattice of an $n$-element set. 
In addition to analogous earlier results in lattice theory,  a connection with cryptology is also mentioned among the motivations. 
\end{abstract}

\dedicatory{Dedicated to the memory of George F.\ McNulty}

\thanks{This research was supported by the National Research, Development and Innovation Fund of Hungary, under funding scheme K 138892.  \hfill{\red{\tbf{\datum}}}}


% 05D05 Extremal set theory
% 06D (1980-now) Distributive lattices


\maketitle

\section{Introduction}\label{sect:intro}
This work belongs both to \emph{extremal combinatorics} and  \emph{lattice theory}. The paper is more or less self-contained; those familiar with  M.Sc.\ level mathematics and the concept of distributive lattices should not have  any difficulty in reading. 

The search for small generating sets has belonged to  lattice theory for long; for example, in chronological order, see Gelfand and Ponomarev \cite{gelfand}, Strietz \cite{strietz}, Z\'adori \cite{zadori1,zadori2}, Chajda and Cz\'edli \cite{ivanczg}, Tak\'ach \cite{takach}, Kulin \cite{kulin}, Cz\'edli and Oluoch \cite{czglo}, and Ahmed and Cz\'edli \cite{delbrinczg}. See also the  surveying parts and the bibliographic sections in \cite{delbrinczg} and Cz\'edli \cite{czgDEBRauth} for further references. If a large lattice $L$ can be generated by few elements, then this lattice has many small generating sets.  Cz\'edli \cite{czgDEBRauth} and \cite{czgboolegen} have recently observed that these lattices can be used for cryptography; for a further note on this topic, see Remark \ref{rem:kCpgCl}. 
This fact and the results on small generating sets of lattices   in the above-mentioned  and some additional papers constitute the \emph{lattice theoretic motivation} of the paper.  

There is a \emph{motivation} coming from \emph{extremal combinatorics}, too. The first result on the maximum number $\Sp(U,n)$ of pairwise unrelated (in other words, incomparable) copies of a poset $U$ in the powerset lattice of an $n$-element finite set was published by  Sperner \cite{sperner} ninety-five years ago. While $U$ is the singleton poset in Sperner's theorem, the \emph{Sperner theorem} (that is, the Sperner \emph{type} theorem) in Griggs, Stahl, and Trotter \cite{griggsatall} determines $\Sp(U,n)$ for any finite chain $U$. For some other finite posets, similar results were obtained by Katona and Nagy \cite{KatonaNagy} and Cz\'edli \cite{czgsp}. In general, the exact value of   $\Sp(U,n)$ is rarely known. On the other hand, Katona and Nagy \cite{KatonaNagy} and, independently from them, Dove and Griggs \cite{dovegriggs} determined the \emph{asymptotic} value of $\Sp(U,n)$. Their celebrated result asserts that for any finite poset $U$, 
\begin{equation}
\Sp(U,n)\sim \frac 1{|U|}\binom n{\lint{n/2}},\text{ that is, }\lim_{n\to\infty} \frac 1{|U|}\binom n{\lint{n/2}}\cdot \Sp(U,n)^{-1} = 1.
\label{eq:DGKNasymp}
\end{equation}

By the main result of \cite{czgsp}, the  lattice theoretic motivation and the combinatorial one are strongly connected; see \eqref{eq:zkRspHjTKc} later, which we are going to quote from  \cite{czgsp}. Here we only mention that in order to get closer to what the title of the paper promises, we need to determine $\Sp(U,n)$ for some rather special posets $U$. 

The asymptotic result  \eqref{eq:DGKNasymp} may suggest that for our special posets $U$, we can obtain $\Sp(U,n)$ or at least some of its estimates simply by copying what  Dove and Griggs \cite{dovegriggs} or Katona and Nagy \cite{KatonaNagy} did. However, we have three reasons not to follow this plan. 
First, while several constructions and considerations can lead to the asymptotically same result, we cannot expect  a similar experience when dealing with small values of $n$. Furthermore, concrete (non-asymptotic) calculations and considerations are often harder and their asymptotic counterparts do not offer too much help. For example, while we know for any fixed $a,b\in \mathbb Z$ (the set of integers)  that, with our vertical-space-saving permanent notation $\fsp(n):=\binom n{\lint{n/2}}$,
\begin{equation}
\binom{n+a}{\lint{n/2}+b}\sim 2^a\cdot \binom n{\lint{n/2}}=2^a\fsp(n)\quad \text{ as } n\to \infty
\label{eq:vZvsznRzbnLjdnkpst}
\end{equation}
and so we can simply work with $2^a\fsp(n)$ in asymptotic considerations, we have to work with $\binom{n+a}{\lint{n/2}+b}$ in  concrete calculations, which is more difficult .  (Note at this point that both  Dove and Griggs \cite{dovegriggs} and Katona and Nagy \cite{KatonaNagy} use \eqref{eq:vZvsznRzbnLjdnkpst}.)
Second, even though a general construction could be specialized to our particular posets $U$, we cannot expect to exploit the peculiarities of our $U$'s in this way. 
Third, an easy-to-read construction with a short and easy argument will hopefully be interesting for the reader, partially because these details are necessary to explain and perform the computations.

Hence, the construction we are going to give for lower estimates is different from those in  Dove and Griggs \cite{dovegriggs} and  Katona and Nagy \cite{KatonaNagy}. At some places in  the proofs, we are going to point out the difference from \cite{dovegriggs}; the difference from  \cite{KatonaNagy} is clearer. Note that our construction gives better lower estimates for our particular posets $U$ than any of the  Dove-Griggs and the Katona-Nagy construction would give, at least for small values of $n$. (For $n\to\infty$, that is, asymptotically, all the three constructions yield the same lower estimate.)
On the other hand, let us emphasize the similarities. While many calculations in this paper are new, most of the ideas in our construction occur in Dove and Griggs \cite{dovegriggs} and  Katona and Nagy \cite{KatonaNagy}; more details will be mentioned right after the proof of Proposition \ref{prop:hYmrbjkSbCklsh}.  

Even though our result allows a big gap between the lower estimate and the upper estimate of $\Sp(U,n)$, this result will suffice to determine the least number $n$ of elements that generate the direct powers $\FD 3^k$ of $\FD 3$ with quite a good accuracy, and we can give reasonable  estimates on $n$ in case of $\FD r^k$. 

\begin{figure}[hbt]
\begin{tikzpicture}[scale=0.5]
\draw (2*\stp,0) coordinate (o);
\draw (\stp,\stp) coordinate (a);
\draw (2*\stp,\stp) coordinate (b);
\draw (3*\stp,\stp) coordinate (c);
\draw (\stp,2*\stp) coordinate (d);
\draw (2*\stp,2*\stp) coordinate (e);
\draw (3*\stp,2*\stp) coordinate (f);
\draw (0,3*\stp) coordinate (x);
\draw (2*\stp,3*\stp) coordinate (g);
\draw (3*\stp,3*\stp) coordinate (y);
\draw (4*\stp,3*\stp) coordinate (z);
\draw (\stp,4*\stp) coordinate (h);
\draw (2*\stp,4*\stp) coordinate (j);
\draw (3*\stp,4*\stp) coordinate (k);
\draw (\stp,5*\stp) coordinate (l);
\draw (2*\stp,5*\stp) coordinate (m);
\draw (3*\stp,5*\stp) coordinate (n);
\draw (2*\stp,6*\stp) coordinate (i);
 \draw (4*\stp+\gap,\hos) coordinate (A);
 \draw (4*\stp+\gap+\bstp,\hos) coordinate (B);
 \draw (4*\stp+\gap+2*\bstp,\hos) coordinate (C);
 \draw (4*\stp+\gap,\hos+\bstp) coordinate (X);
 \draw (4*\stp+\gap+\bstp,\hos+\bstp) coordinate (Y);
 \draw (4*\stp+\gap+2*\bstp,\hos+\bstp) coordinate (Z);
\draw[thick](o)--(a); \draw[thick](o)--(b); \draw[thick](o)--(c);
\draw[thick](a)--(d); \draw[thick](a)--(e); \draw[thick](b)--(d); \draw[thick](b)--(f); \draw[thick](c)--(e); \draw[thick](c)--(f);
\draw[thick](d)--(x); \draw[thick](d)--(g); \draw[thick](e)--(y); \draw[thick](e)--(g); \draw[thick](f)--(z); \draw[thick](f)--(g); \draw[thick](x)--(h); \draw[thick](g)--(h); \draw[thick](g)--(j); \draw[thick](g)--(k); \draw[thick](y)--(j); \draw[thick](z)--(k); 
\draw[thick](h)--(l);\draw[thick](h)--(m);\draw[thick](j)--(l);\draw[thick](j)--(n);\draw[thick](k)--(m);\draw[thick](k)--(n); 
\draw[thick](l)--(i);\draw[thick](m)--(i);\draw[thick](n)--(i);
 \draw[thick](A)--(X); \draw[thick](A)--(Y); \draw[thick](B)--(X); \draw[thick](B)--(Z); \draw[thick](C)--(Y); \draw[thick](C)--(Z);
 \bcirc a; \bcirc b; \bcirc c; \bcirc x; \bcirc y; \bcirc z;
\wcirc o; \wcirc d; \wcirc e; \wcirc f; \wcirc g; \wcirc g; \wcirc h; \wcirc j; \wcirc k; \wcirc l; \wcirc m; \wcirc n; \wcirc i
; \wcirc A; \wcirc B; \wcirc C; \wcirc X; \wcirc Y; \wcirc Z;
\nodetextr x{$x$}; \nodetextl y{$y$}; \nodetextl z{$z$}; \draw (2.5*\stp,5.8*\stp) coordinate (LL); \nodetextr{LL}{$\FD3$};
\draw (3.84*\stp+ \gap, \hos + 2*\bstp) coordinate (RL); \nodetextr{RL}{$\FSP303$};
\nodetextb{A}{$A$}; \nodetextb{B}{$B$}; \nodetextb{C}{$C$};
\nodetexta{X}{$X$};\nodetexta{Y}{$Y$};\nodetexta{Z}{$Z$};
\nodetextl{a}{$a$};\nodetextl{b}{$b$};\nodetextr{c}{$c$};
\end{tikzpicture}
\caption{$\FD3$ and the $3$-crown $\crn=\FSP303\cong\Jir{\FD3}$}
\label{figone}
\end{figure}

\section{Basic facts and notations}
Except for $\Nplu:=\set{1,2,3,\dots}$, $\Nnul:=\set0\cup\Nplu$,  $\Nfromthree:=\set{3,4,5,\dots}=\Nplu\setminus\set{1,2}$  and their subsets, all sets and structures  in the paper will be assumed to be \emph{finite}. (Sometimes, we repeat this convention for those who read only a part of the paper.) For $r\in\Nfromwhat 3$, the \emph{free distributive lattice on $r$  generators} is denoted by $\FD r$; for $r=3$, it is drawn on the left of Figure \ref{figone}. A lattice element with exactly one lower cover is called \emph{join-irreducible}. For a lattice $L$, the \emph{poset} (that is, the \emph{\acr partially  \acr ordered \acr{set}}) of the \acr{j}oin-irreducible elements of $L$ is denoted by $\Jir L$. For  $L=\FD3$, $\Jir L$ consists of the black-filled elements and it is also drawn separately on the right of the figure. For a set $H$, the \emph{\acr{pow}erset lattice} of $H$ is $(\set{Y:Y\subseteq H};\cup,\cap)$; it (or its support set) is denoted by $\Pow H$. For $n\in\Nnul$, the set $\set{1,2,\dots,n}$ is denoted by $[n]$; note that $[0]= \emptyset$. For $x,y$ in a poset, in particular, for $x,y\in\nPow$, we write $x\parallel y$ to denote that neither $x\leq y$ nor $y\leq x$  holds; in $\nPow$, ``$\leq$'' is ``$\subseteq$''.
For a poset $U$,  a \emph{copy} of $U$ in $\nPow$ is a subset of $\nPow$ that, equipped with ``$\subseteq$'', is order isomorphic to $U$. Two copies of $U$ in $\nPow$ are \emph{unrelated} if for all  $X$ in the first copy and all $Y$ in the second copy, $X\parallel Y$. Let us repeat that for $n\in\Nnul$ and a poset $U$, we let
\begin{equation}
\begin{aligned}
\Sp(U,n):=\max\{k:{} &\text{there exist }k\text{ pairwise}\cr
&\text{unrelated copies of }U\text{ in }\nPow\}.
\end{aligned}
\label{eq:Spnotat}
\end{equation}
According to the sentence containing \eqref{eq:vZvsznRzbnLjdnkpst}, we often write $\ibinom n k$ instead of $\binom n k$; especially in text environment and  if $n$ or $k$ are complicated or subscripted expressions. 
The notation``$\Sp(-,-)$'' and ``$\ibinom--$''  come from \acr{Sp}erner and \acr{b}inomial \acr{c}oefficient, respectively. As usual, $\lint{\,\,\,}$ and $\uint{\,\,\,}$ denote the  lower and upper \emph{integer part} functions;  for example, $\lint{5/3}=1$ and $\uint{5/3}=2$. With our notations, Sperner's theorem \cite{sperner} asserts that for every $n\in\Nnul$, 
\begin{equation}
\text{if  }U\text{ is the $1$-element poset, then }\Sp(U,n)=\binom{n}{\lint{n/2}}=:\fsp(n).
\label{eq:fspnotat}
\end{equation}

Recall that a subset $X$ a lattice $L=(L;\vee,\wedge)$  is a \emph{generating set} of $L$ if for every $Y$ such that $X\subseteq Y\subseteq L$ and $Y$ is closed with respect to $\vee$ and $\wedge$, we have that $Y=L$. We denote the \emph{size of a \acr{m}inimum-sized \acr{g}enerating set} of $L$ by 
\begin{equation}
\gmin L:=\min\set{|X|: X\text{ is a generating set of }L}.
\label{eq:gmin}
\end{equation}
For $k\in\Nplu$, the $k$-th \emph{direct power} $L^k$ of $L$ consists of the $k$-tuples of elements of $L$ and the lattice operations are performed componentwise. With  our notations, the main result of Cz\'edli \cite{czgsp} asserts that 
\begin{equation}
\parbox{9cm}{for $2\leq k\in\Nplu$ and a finite distributive lattice $L$, $\gmin {L^k}$ is the smallest $n\in\Nplu$ such that $k\leq \Sp(\Jir L,n)$.}
\label{eq:zkRspHjTKc}
\end{equation}
It is also clear from \cite{czgsp} that for each finite distributive lattice $L$, the functions $k\mapsto \gmin{L^k}$ and $n\mapsto \Sp(\Jir L,n)$ mutually determine each other, but we do not need this fact in the present paper. The following definition is crucial in the paper.

\begin{definition}\label{def:flp} For $0\leq a<b\leq r\in \Nnul$ such that $a+2\leq b$, the \emph{\acr{f}ull \acr{s}egment \acr{p}oset} $\FSP r a b$ is the poset $U$ defined (up to isomorphism) by the conjunction of the following two rules.
\begin{enumerate}
\item[(a)] $r$ is the smallest integer such that $U$ is embeddable into $\Pow{[r]}$;
\item[(b)] the subposet $\set{X\in\Pow{[r]}: a<|X|<b}$ of $\Pow{[r]}$ is order isomorphic to $U$.
\end{enumerate}
\end{definition}

Even though $0\leq a$ in Definition \ref{def:flp} could be replaced by by $-1\leq a$, we do not do so since the case $a=-1$  would need a different (in fact, easier) treatment; see \cite{czgsp}.
Let $U$ be a finite poset, $s\in\Nplu$, and denote $\set{s,s+1,s+2,\dots}$ by $\Nfromwhat s$. 
If $f_1,f_2\colon \Nfromwhat s\to \Nnul$ are functions such that $f_1(n)\leq \Sp(U,n)\leq f_2(n)$ for all $n\in\Nfromwhat s$, then $(f_1,f_2)$ is a \emph{pair of estimates} of the function  $\Sp(U,-)$ on $\Nfromwhat s$; in particular,   $f_1$ is a \emph{lower estimate} while $f_2$ is an  \emph{upper estimate} of $\Sp(U,-)$.  A reasonably good property of pairs of estimates of $\Sp(U,-)$ is defined as follows: 
\begin{equation}
\parbox{7.7cm}{for $s\in\Nplu$, a pair $(f_1,f_2)$ of  estimates is \emph{separated} on $\Nfromwhat s$  if $f_2(n) \leq f_1(n+1)$ for all $n\in\Nfromwhat s$.}
\label{eq:fwftspRtd}
\end{equation}
The following fact is a trivial consequence of \eqref{eq:zkRspHjTKc} and for $k\geq 2$, it  is implicit in Cz\'edli \cite{czgsp}; see around (5.23) and (5.24) in \cite{czgsp}.%

\begin{observation}\label{obs:nthNmb} Let $D$ be a finite distributive lattice. Denote the poset $\Jir D$ by $U$, and let  $s\in\Nplu$. Let $(f_1,f_2)$ be a separated pair of  estimates of $\Sp(U,-)$ on $\Nfromwhat s$  such that   $f_1$  (the lower estimate) is   strictly increasing on $\Nfromwhat s$ . Then,  for each  $k\in \Nplu$ such that   $f_1(s)< k$,  $(f_1,f_2)$ determines $\gmin{D^k}$  ``\emph{with accuracy $1/2$}'' as follows: Letting $n$ be  the unique $n\in\Nplu$ such that $f_1(n)<k\leq f_1(n+1)$,   either $k\leq f_2(n)$ and $\gmin{D^k}\in\set{n,n+1}$  or $f_2(n)<k$ and  $\gmin{D^k}=n+1$.
\end{observation}

The term ``accuracy $1/2$'' comes from the fact that the distance between the never exact estimate $n+1/2$  and $\gmin{D^k}$ is always $1/2$.

\section{Lower estimates}

The easy proof of the following lemma  raises the possibility that the lemma might belong to the folklore even though the author has never met it.

\begin{lemma} \label{lemma:wDwdRjRd}
For $2\leq r\in\Nplu$, $\Jir {\FD r}\cong \FSP r 0 r$; see Definition \ref{def:flp}.
\end{lemma}

\begin{proof} Denote by  $\set{x_1,\dots,x_r}$ the set of free generators of $\FD r$. Call a subset $J$ of $[r]$ \emph{nontrivial} if $\emptyset\neq J\neq [r]$, and let $\ntPow{[r]}=\bigl(\ntPow{[r]};\subseteq\bigr)$ stand for the poset formed by the  nontrivial subsets of $[r]$. For $J\in \ntPow{[r]}$, let $m_J$ be the meet  $\bigwedge_{i\in J}x_i$, and define $X:=\set{m_J: J\in \ntPow{[r]}}$. As $X\subseteq \FD r$, $X=(X;\leq)$ is a subposet of $\FD r$. 

First, we show that the map $\phi\colon \ntPow{[r]}\to X$ defined by $J\to m_J$ is a dual order isomorphism.  The tool wee need is very simple: Since $\FD r$ is free, it follows that whenever $J,K\in\ntPow{[r]}$ and   $m_J=m_K$, then $m_J(\vec y)=m_K(\vec y)$ for all  $\vec y=(y_1,\dots,y_r)\in\set{0,1}^r$, and similarly for ``$\leq$'' instead of ``$=$''. The implication $J\subseteq K\Rightarrow m_J\geq m_K$ is obvious. For the sake of contradiction, suppose that $m_J\geq m_K$ for some $J,K\in\ntPow{[r]}$ but $J\nsubseteq K$. Pick a $j\in J\setminus K$, and let $\vec y\in \set{0,1}^r$ be the vector for which $y_j=0$ but $y_i=1$ for all $i\in[r]\setminus\set j$.  Then $m_K=\vec y =1$ since the $j$-th component of $\vec y$ does not occur in the meet but $m_J=0$, contradicting $m_J\geq m_K$. This proves that ``$\geq$'' in $X$ and ``$\subseteq$'' in  $\ntPow{[r]}$ correspond to each other. In particular, $\phi$ is a bijective map as the equality of two elements or subsets can be expressed by these relations. Thus, $\phi$ is a dual order isomorphism.  The composite of $\phi$ and the selfdual automorphism of $\ntPow{[r]}$ defined by $J\mapsto [r]\setminus J$ is an order isomorphism, proving that $X\cong \FSP r 0 r$. 

Next, to complete the proof, it suffices to show that $\Jir{\FD r}=X$. Using the tool (with $\vec y$) mentioned earlier, observe that $1=x_1\vee\dots\vee x_r\notin \Jir {\FD r}$ and for every $J\in\ntPow{[r]}$,  $m_J\notin\set{0, 1}$. By distributivity, each element of $\FD r\setminus\set{0,1}$ is the join of meets of some generators or, in other words, a disjunctive normal form of the generators. Clearly,  neither the empty meet, nor the empty join, nor the meet of all generators is needed here, whereby there is at least one joinand and each of the  joinands is of the form $m_J$ with $J\in\ntPow{[r]}$. As one joinand is sufficient for the elements of $\Jir {\FD r}$, we obtain that $\Jir {\FD r}\subseteq X$. 

To show that converse inclusion by way of contradiction, suppose that $m_J\in X\setminus \Jir {\FD r}$. Then $m_J$ is the join of some elements of $\Jir{\FD r}$ that are smaller than $m_J$. 
These elements are of the form $m_{I_j}$ as $\Jir {\FD r}\subseteq X$. This fact and  dual isomorphism proved in the previous paragraph imply that there are $I_1,\dots,I_t\in\ntPow{[r]}$ such that $J\subset I_1$, \dots, $J\subset I_t$ and $m_J=m_{I_1} \vee \dots \vee m_{I_t}$. This equality holds as an identity in the two-element lattice $\set{0,1}$.  However, if we define $\vec y\in \set{0,1}^r$ by $y_s:=1$ if $s\in J$ and $y_s=0$ otherwise, then $m_J(\vec y)=1$ but each of the joinands and so the join are $0$. This contradiction completes the proof.
\end{proof}

For $1\leq a<b\leq r\in \Nplu$ such that $a+2\leq b$ and $n\in \Nfromwhat r$, 
$\vec v$ will denote a vector $(v_0,\dots,v_a;v_b,\dots,v_r)$, so there is gap in the index set of the components.  Let $p\in\set{-r,-r+1,\dots, r}$ be a parameter, and let us agree that a binomial coefficient $\ibinom{x_1}{x_2}$ is 0 unless $x_1,x_2\in\Nnul$ and $0\leq x_2\leq x_1$.  With these conventions, define
 \begin {align}
 &\begin{aligned}
 \fprab p r a b&(n):=\sum_{i=0}^{\lint{n/r}-1} \sum_{{\sct{\vec v\in\set{0,\dots,i}^{r+a-b+2}}\atop{\sct{v_0+\dots+v_a+v_b+ \dots+v_{r}=i}}}}
 \frac{i!}{v_0!\dots v_a! \cdot v_b!\dots  v_r!}{\,\,}\times \cr
 &\times\binom{n-(i+1)r}{p+\lint{(n-r)/2}-0v_0-1v_1-\dots -av_a-bv_b-\dots - rv_r} \times 
 \cr
 &  \times \binom{r}{0}^{\kern -3pt v_0} \dots \binom{r}{a}^{\kern -3pt v_a}
 \cdot  \binom{r}{b}^{\kern -3pt v_b}\dots \binom{r}{r}^{\kern -3pt v_r}, \qquad \text{ and}
\end{aligned} \label{eq:fprab}
\\
& \frab r a b(n):=\max\set{ \fprab p r a b(n): p\in\set{-r,-r+1,\dots,r-1,r}}.
\label{eq:frab}
\end{align}

\begin{proposition}\label{prop:hYmrbjkSbCklsh}
For $r\in\Nfromwhat 3$ and  $0\leq a<b\leq r\in \Nplu$ such that $a+2\leq b$, $\frab r a b(n)$ is a lower estimate of $\Sp(\FSP r a b, n)$ on $\Nfromwhat r$.
\end{proposition}

The proof below shows that  Proposition \ref{prop:hYmrbjkSbCklsh} would still hold if we replaced $\set{-r,-r+1,\dots,r-1,r}$ with $\mathbb Z$ but we do not have any example where $\mathbb Z$, which would make practical computations longer, is better than  $\set{-r,-r+1,\dots,r-1,r}$.

 \begin{proof}   
It suffices to show  that for any $p\in \mathbb Z$,  $\fprab p r a b(n)\leq \Sp(\FSP r a b,n)$.  Take an $n$-element set $M$, and denote the quotient $\lint{n/r}$ by $q$. Fix $q$ pairwise disjoint subsets $M_0$, \dots, $M_{q-1}$ of $M$, we call them \emph{blocks},  and let  $M_q:=M\setminus(M_0\cup\dots\cup M_{q-1})$. Let $h:=p+\lint{(n-r)/2}$. For $j\in \set{0,\dots,q-1}$, a subset $X$ of the block $M_j$ is called \emph{small} if $|X|\leq a$. Similarly, if $|X|\geq b$, then $X$ is \emph{large} while in the remainder case when $a<|X|<b$, we say that $X$ is \emph{medium-sized}. By an \emph{extremal subset} of $M_j$ we mean a subset that is large or small; so ``extremal'' is the opposite of ``medium-sized''. For a subset $B$ of $M$,  $B\cap M_i$ is often denoted by $B_i$. We say that $(i,B)\in\set{0,\dots,q-1}\times \Pow M$ is  a  \emph{fundamental pair} if 
\begin{enumerate}
\item[(F1)] $|B|=h$, and 
\item[(F2)] $B_i=\emptyset $  and for each $j\in\set{0,\dots, i-1}$, $B_j$ is extremal (that is, small or large).
\end{enumerate}
Four examples are given in Figure \ref{figtwo}, where $n=54$, $r=8$, $a=3$, $b=6$, $q=6$, and  $h=26$. In each of the four parts of this figure, the green-filled solid ovals\footnote{Note for a grayscale  version: the green-filled ovals contain  black numbers in their interiors while the ovals with white numbers are magenta-filled.} represent extremal subsets of the appropriate $M_j$'s, $j\in\set{0,\dots,i-1}$,
the red dotted oval is  a medium-sized subset of $M_i$, and there is no condition on the subsets represented by magenta-filled solid ovals. Hence,  in each of the four examples, the \emph{set component} (that is, the second component, which was denoted by $B$)  of the fundamental pair is the union of the color-filled solid ovals. The \emph{index component} (that is, the first component) is indicated at the top of the figure.  Each color-filled solid oval contains the number of elements of the subset $B_j$ that this oval  represents. Note, however, that a red dotted oval (regardless the number it contains) in the picture of $(i,B)$ means that $B_i=\emptyset$.  (The red dotted ovals will be explained right after \eqref{eq:kfbKmGwrjCltr}.) Note also that, witnessed by $i=5$ and $i=4$ in the figure, the set component does not determine the index component.

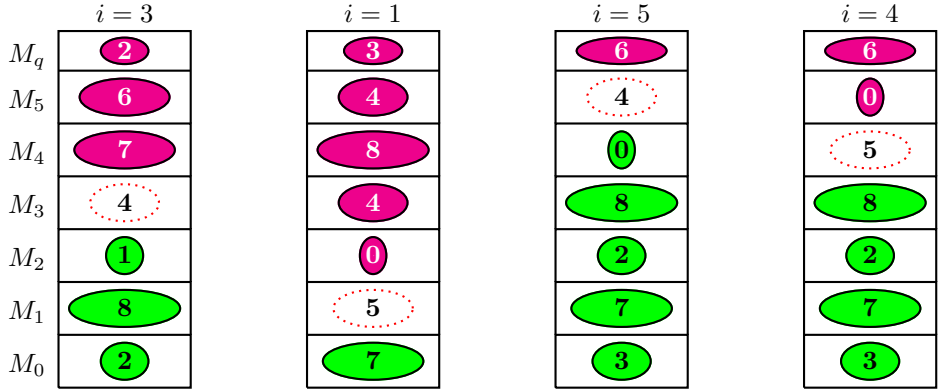
\begin{figure}[hbt]
\begin{tikzpicture}[scale=1.0]
% Left labels
\draw(0.6*\hot, 1.0*\vot+0*\vot) coordinate (L0);
\draw(0.6*\hot, 1.0*\vot+1*\vot) coordinate (L1);
\draw(0.6*\hot, 1.0*\vot+2*\vot) coordinate (L2);
\draw(0.6*\hot, 1.0*\vot+3*\vot) coordinate (L3);
\draw(0.6*\hot, 1.0*\vot+4*\vot) coordinate (L4);
\draw(0.6*\hot, 1.0*\vot+5*\vot) coordinate (L5);
\draw(0.6*\hot, 1.0*\vot+5.80*\vot) coordinate (L6);
\nodeuutxt{L0}{$M_0$};
\nodeuutxt{L1}{$M_1$};
\nodeuutxt{L2}{$M_2$};
\nodeuutxt{L3}{$M_3$};
\nodeuutxt{L4}{$M_4$};
\nodeuutxt{L5}{$M_5$};
\nodeuutxt{L6}{$M_q$};
% 0-th box:
\draw(\hot,\vot) coordinate (a);
\draw(\hot+\wit,\vot) coordinate (b);
\draw(\hot,\vot+\het) coordinate (c);
\draw(\hot+\wit,\vot+\het) coordinate (d);
\draw(\hot,\vot+2*\het) coordinate (e);
\draw(\hot+\wit,\vot+2*\het) coordinate (f);
\draw(\hot,\vot+3*\het) coordinate (g);
\draw(\hot+\wit,\vot+3*\het) coordinate (h);
\draw(\hot,\vot+4*\het) coordinate (i);
\draw(\hot+\wit,\vot+4*\het) coordinate (j);
\draw(\hot,\vot+5*\het) coordinate (k);
\draw(\hot+\wit,\vot+5*\het) coordinate (l);
\draw(\hot,\vot+6*\het) coordinate (uk);
\draw(\hot+\wit,\vot+6*\het) coordinate (ul);
\draw(\hot,\vot+6.75*\het) coordinate (m);
\draw(\hot+0.5*\wit,\vot+6.80*\het) coordinate (s);
\draw(\hot+\wit,\vot+6.75*\het) coordinate (n);
\draw(\hot+0.5*\wit,\vot+0.5*\het) coordinate (ell0);
\draw(\hot+0.5*\wit,\vot+1.5*\het) coordinate (ell1);
\draw(\hot+0.5*\wit,\vot+2.5*\het) coordinate (ell2);
\draw(\hot+0.5*\wit,\vot+3.5*\het) coordinate (ell3);
\draw(\hot+0.5*\wit,\vot+4.5*\het) coordinate (ell4);
\draw(\hot+0.5*\wit,\vot+5.5*\het) coordinate (ell5);
\draw(\hot+0.5*\wit,\vot+6.375*\het) coordinate (ell6);
\draw[thick,green,fill] (ell0) ellipse (9pt and 7pt);
\draw[thick] (ell0) ellipse (9pt and 7pt);
\nodeelltxt{ell0}{2};
\draw[thick,green,fill] (ell1) ellipse (21pt and 7pt);
\draw[thick] (ell1) ellipse (21pt and 7pt);
\nodeelltxt{ell1}{8};
\draw[thick,green,fill] (ell2) ellipse (7pt and 7pt);
\draw[thick] (ell2) ellipse (7pt and 7pt);
\nodeelltxt{ell2}{1};
\draw[thick,dotted,red] (ell3) ellipse (13pt and 7pt);
\nodeelltxt{ell3}{4};
\draw[thick,magenta,fill] (ell4) ellipse (19pt and 7pt);
\draw[thick] (ell4) ellipse (19pt and 7pt);
\nodewhtxt{ell4}{7};
\draw[thick,magenta,fill] (ell5) ellipse (17pt and 7pt);
\draw[thick] (ell5) ellipse (17pt and 7pt);
\nodewhtxt{ell5}{6};
\draw[thick,magenta,fill] (ell6) ellipse (9pt and 5pt);
\draw[thick] (ell6) ellipse (9pt and 5pt);
\nodewhtxt{ell6}{2};
\draw[thick] (a)--(b)--(d)--(c)--(a);
\draw[thick]  (c)--(e)--(f)--(d);
\draw[thick]  (e)--(g)--(h)--(f);
\draw[thick]  (g)--(i)--(j)--(h);
\draw[thick]  (i)--(k)--(l)--(j);
\draw[thick]  (k)--(m)--(n)--(l);
\draw[thick]  (uk)--(ul);
%\lcirc {o};\lcirc{q};\lcirc{p};\lcirc{r};
\nodeutxt{s}{$i=3$};
% End of 0-th box
%
%
% 1-st box:
\draw(1*\whh+\hot,\vot) coordinate (a);
\draw(1*\whh+\hot+\wit,\vot) coordinate (b);
\draw(1*\whh+\hot,\vot+\het) coordinate (c);
\draw(1*\whh+\hot+\wit,\vot+\het) coordinate (d);
\draw(1*\whh+\hot,\vot+2*\het) coordinate (e);
\draw(1*\whh+\hot+\wit,\vot+2*\het) coordinate (f);
\draw(1*\whh+\hot,\vot+3*\het) coordinate (g);
\draw(1*\whh+\hot+\wit,\vot+3*\het) coordinate (h);
\draw(1*\whh+\hot,\vot+4*\het) coordinate (i);
\draw(1*\whh+\hot+\wit,\vot+4*\het) coordinate (j);
\draw(1*\whh+\hot,\vot+5*\het) coordinate (k);
\draw(1*\whh+\hot+\wit,\vot+5*\het) coordinate (l);
\draw(1*\whh+\hot,\vot+6*\het) coordinate (uk);
\draw(1*\whh+\hot+\wit,\vot+6*\het) coordinate (ul);
\draw(1*\whh+\hot,\vot+6.75*\het) coordinate (m);
\draw(1*\whh+\hot+0.5*\wit,\vot+6.80*\het) coordinate (s);
\draw(1*\whh+\hot+\wit,\vot+6.75*\het) coordinate (n);
\draw(1*\whh+\hot+0.5*\wit,\vot+0.5*\het) coordinate (ell0);
\draw(1*\whh+\hot+0.5*\wit,\vot+1.5*\het) coordinate (ell1);
\draw(1*\whh+\hot+0.5*\wit,\vot+2.5*\het) coordinate (ell2);
\draw(1*\whh+\hot+0.5*\wit,\vot+3.5*\het) coordinate (ell3);
\draw(1*\whh+\hot+0.5*\wit,\vot+4.5*\het) coordinate (ell4);
\draw(1*\whh+\hot+0.5*\wit,\vot+5.5*\het) coordinate (ell5);
\draw(1*\whh+\hot+0.5*\wit,\vot+6.375*\het) coordinate (ell6);
\draw[thick,green,fill] (ell0) ellipse (19pt and 7pt);
\draw[thick] (ell0) ellipse (19pt and 7pt);
\nodeelltxt{ell0}{7};
\draw[thick,dotted,red] (ell1) ellipse (15pt and 7pt);
\nodeelltxt{ell1}{5};
\draw[thick,magenta,fill] (ell2) ellipse (5pt and 7pt);
\draw[thick] (ell2) ellipse (5pt and 7pt);
\nodewhtxt{ell2}{0};
\draw[thick,magenta,fill] (ell3) ellipse (13pt and 7pt);
\draw[thick] (ell3) ellipse (13pt and 7pt);
\nodewhtxt{ell3}{4};
\draw[thick,magenta,fill] (ell4) ellipse (21pt and 7pt);
\draw[thick] (ell4) ellipse (21pt and 7pt);
\nodewhtxt{ell4}{8};
\draw[thick,magenta,fill] (ell5) ellipse (13pt and 7pt);
\draw[thick] (ell5) ellipse (13pt and 7pt);
\nodewhtxt{ell5}{4};
\draw[thick,magenta,fill] (ell6) ellipse (11pt and 5pt);
\draw[thick] (ell6) ellipse (11pt and 5pt);
\nodewhtxt{ell6}{3};
\draw[thick] (a)--(b)--(d)--(c)--(a);
\draw[thick]  (c)--(e)--(f)--(d);
\draw[thick]  (e)--(g)--(h)--(f);
\draw[thick]  (g)--(i)--(j)--(h);
\draw[thick]  (i)--(k)--(l)--(j);
\draw[thick]  (k)--(m)--(n)--(l);
\draw[thick]  (uk)--(ul);
%\lcirc {o};\lcirc{q};\lcirc{p};\lcirc{r};
\nodeutxt{s}{$i=1$};
% End of 1-st box
%
%
%
% 2nd box:
\draw(2*\whh+\hot,\vot) coordinate (a);
\draw(2*\whh+\hot+\wit,\vot) coordinate (b);
\draw(2*\whh+\hot,\vot+\het) coordinate (c);
\draw(2*\whh+\hot+\wit,\vot+\het) coordinate (d);
\draw(2*\whh+\hot,\vot+2*\het) coordinate (e);
\draw(2*\whh+\hot+\wit,\vot+2*\het) coordinate (f);
\draw(2*\whh+\hot,\vot+3*\het) coordinate (g);
\draw(2*\whh+\hot+\wit,\vot+3*\het) coordinate (h);
\draw(2*\whh+\hot,\vot+4*\het) coordinate (i);
\draw(2*\whh+\hot+\wit,\vot+4*\het) coordinate (j);
\draw(2*\whh+\hot,\vot+5*\het) coordinate (k);
\draw(2*\whh+\hot+\wit,\vot+5*\het) coordinate (l);
\draw(2*\whh+\hot,\vot+6*\het) coordinate (uk);
\draw(2*\whh+\hot+\wit,\vot+6*\het) coordinate (ul);
\draw(2*\whh+\hot,\vot+6.75*\het) coordinate (m);
\draw(2*\whh+\hot+0.5*\wit,\vot+6.80*\het) coordinate (s);
\draw(2*\whh+\hot+\wit,\vot+6.75*\het) coordinate (n);
\draw(2*\whh+\hot+0.5*\wit,\vot+0.5*\het) coordinate (ell0);
\draw(2*\whh+\hot+0.5*\wit,\vot+1.5*\het) coordinate (ell1);
\draw(2*\whh+\hot+0.5*\wit,\vot+2.5*\het) coordinate (ell2);
\draw(2*\whh+\hot+0.5*\wit,\vot+3.5*\het) coordinate (ell3);
\draw(2*\whh+\hot+0.5*\wit,\vot+4.5*\het) coordinate (ell4);
\draw(2*\whh+\hot+0.5*\wit,\vot+5.5*\het) coordinate (ell5);
\draw(2*\whh+\hot+0.5*\wit,\vot+6.375*\het) coordinate (ell6);
\draw[thick,green,fill] (ell0) ellipse (11pt and 7pt);
\draw[thick] (ell0) ellipse (11pt and 7pt);
\nodeelltxt{ell0}{3};
\draw[thick,green,fill] (ell1) ellipse (19pt and 7pt);
\draw[thick] (ell1) ellipse (19pt and 7pt);
\nodeelltxt{ell1}{7};
\draw[thick,green,fill] (ell2) ellipse (9pt and 7pt);
\draw[thick] (ell2) ellipse (9pt and 7pt);
\nodeelltxt{ell2}{2};
\draw[thick,green,fill] (ell3) ellipse (21pt and 7pt);
\draw[thick] (ell3) ellipse (21pt and 7pt);
\nodeelltxt{ell3}{8};
\draw[thick,green,fill] (ell4) ellipse (5pt and 7pt);
\draw[thick] (ell4) ellipse (5pt and 7pt);
\nodeelltxt{ell4}{0};
\draw[thick,dotted,red] (ell5) ellipse (13pt and 7pt);
\nodeelltxt{ell5}{4};
\draw[thick,magenta,fill] (ell6) ellipse (17pt and 5pt);
\draw[thick] (ell6) ellipse (17pt and 5pt);
\nodewhtxt{ell6}{6};
\draw[thick] (a)--(b)--(d)--(c)--(a);
\draw[thick]  (c)--(e)--(f)--(d);
\draw[thick]  (e)--(g)--(h)--(f);
\draw[thick]  (g)--(i)--(j)--(h);
\draw[thick]  (i)--(k)--(l)--(j);
\draw[thick]  (k)--(m)--(n)--(l);
\draw[thick]  (uk)--(ul);
%\lcirc {o};\lcirc{q};\lcirc{p};\lcirc{r};
\nodeutxt{s}{$i=5$};
% End of 2nd box
%
%
%
%
% 3rd box:
\draw(3*\whh+\hot,\vot) coordinate (a);
\draw(3*\whh+\hot+\wit,\vot) coordinate (b);
\draw(3*\whh+\hot,\vot+\het) coordinate (c);
\draw(3*\whh+\hot+\wit,\vot+\het) coordinate (d);
\draw(3*\whh+\hot,\vot+2*\het) coordinate (e);
\draw(3*\whh+\hot+\wit,\vot+2*\het) coordinate (f);
\draw(3*\whh+\hot,\vot+3*\het) coordinate (g);
\draw(3*\whh+\hot+\wit,\vot+3*\het) coordinate (h);
\draw(3*\whh+\hot,\vot+4*\het) coordinate (i);
\draw(3*\whh+\hot+\wit,\vot+4*\het) coordinate (j);
\draw(3*\whh+\hot,\vot+5*\het) coordinate (k);
\draw(3*\whh+\hot+\wit,\vot+5*\het) coordinate (l);
\draw(3*\whh+\hot,\vot+6*\het) coordinate (uk);
\draw(3*\whh+\hot+\wit,\vot+6*\het) coordinate (ul);
\draw(3*\whh+\hot,\vot+6.75*\het) coordinate (m);
\draw(3*\whh+\hot+0.5*\wit,\vot+6.80*\het) coordinate (s);
\draw(3*\whh+\hot+\wit,\vot+6.75*\het) coordinate (n);
\draw(3*\whh+\hot+0.5*\wit,\vot+0.5*\het) coordinate (ell0);
\draw(3*\whh+\hot+0.5*\wit,\vot+1.5*\het) coordinate (ell1);
\draw(3*\whh+\hot+0.5*\wit,\vot+2.5*\het) coordinate (ell2);
\draw(3*\whh+\hot+0.5*\wit,\vot+3.5*\het) coordinate (ell3);
\draw(3*\whh+\hot+0.5*\wit,\vot+4.5*\het) coordinate (ell4);
\draw(3*\whh+\hot+0.5*\wit,\vot+5.5*\het) coordinate (ell5);
\draw(3*\whh+\hot+0.5*\wit,\vot+6.375*\het) coordinate (ell6);
\draw[thick,green,fill] (ell0) ellipse (11pt and 7pt);
\draw[thick] (ell0) ellipse (11pt and 7pt);
\nodeelltxt{ell0}{3};
\draw[thick,green,fill] (ell1) ellipse (19pt and 7pt);
\draw[thick] (ell1) ellipse (19pt and 7pt);
\nodeelltxt{ell1}{7};
\draw[thick,green,fill] (ell2) ellipse (9pt and 7pt);
\draw[thick] (ell2) ellipse (9pt and 7pt);
\nodeelltxt{ell2}{2};
\draw[thick,green,fill] (ell3) ellipse (21pt and 7pt);
\draw[thick] (ell3) ellipse (21pt and 7pt);
\nodeelltxt{ell3}{8};
\draw[thick,dotted,red] (ell4) ellipse (15pt and 7pt);
\nodeelltxt{ell4}{5};
\draw[thick,magenta,fill] (ell5) ellipse (5pt and 7pt);
\draw[thick] (ell5) ellipse (5pt and 7pt);
\nodewhtxt{ell5}{0};
\draw[thick,magenta,fill] (ell6) ellipse (17pt and 5pt);
\draw[thick] (ell6) ellipse (17pt and 5pt);
\nodewhtxt{ell6}{6};
\draw[thick] (a)--(b)--(d)--(c)--(a);
\draw[thick]  (c)--(e)--(f)--(d);
\draw[thick]  (e)--(g)--(h)--(f);
\draw[thick]  (g)--(i)--(j)--(h);
\draw[thick]  (i)--(k)--(l)--(j);
\draw[thick]  (k)--(m)--(n)--(l);
\draw[thick]  (uk)--(ul);
%\lcirc {o};\lcirc{q};\lcirc{p};\lcirc{r};
\nodeutxt{s}{$i=4$};
% End of 3rd box
%
%
%
\end{tikzpicture}
\caption{Illustrating the proof of Proposition \ref{prop:hYmrbjkSbCklsh} with $\FSP 836$; $h=26$, $n=54$; in each fundamental pair, the set component is the union of the color-filled solid ovals.}
\label{figtwo}
\end{figure}

For a fundamental pair $(i,B)$, let 
\begin{equation}
U(i,B):=\set{B\cup X:  X\subseteq M_i\text{ and } a<|X|<b}.
\label{eq:kfbKmGwrjCltr}
\end{equation}
Clearly, $U(i,B)$ is a copy of $\FSP r a b$. The role of a red dotted oval in  Figure \ref{figtwo} is to represent one of the sets  $X$ in \eqref{eq:kfbKmGwrjCltr}.  Now that we have defined our construction, we have to prove that the number of fundemental pairs is $\fprab  p r a b(n)$ and for different fundamental pairs $(i,B)$ and $(i',B')$, $U(i,B)$ and $U(i',B')$ are unrelated.

To obtain a fundamental pair $(i,B)$, first we choose $i\in\set{0,\dots,q-1}$; this explains the outer summation sign in \eqref{eq:fprab}. Then for each $j\in\set{0,\dots,a,b,\dots, r}$ we chose the number $v_j$ of the $j$-element green-filled solid ovals. As there are $i$ green-filled solid ovals, the choice of the vector formed from these $v_j$'s  is not quite arbitrary; this explains  the subscript of the inner summation sign in  \eqref{eq:fprab}.  For example, on the right (that is, in the $i=4$ part) of Figure \ref{figtwo}, $\vec v=(v_0,\dots,v_3;v_6,v_7,v_8)=(0,0,1,1;0,1,1)$. The fraction in  \eqref{eq:fprab} is the multinomial coefficient showing  how many ways $v_0$ zeros, $v_1$  1's, \dots, $v_a$ $a$'s, $v_b$ $b's$, \dots, $v_r$  $r$'s can be ordered. On the right of the figure, this is how many ways the numbers 3, 7, 2,  8 can be written below the red dotted oval (the figure shows only one of these ways).  As there is no stipulation on the magenta-filled solid ovals, the 
binomal coefficient in the middle of  \eqref{eq:fprab} gives the number of possible unions of the magenta-filled solid ovals, that it, it shows  how many ways the system of these ovals can be chosen.

For $j\in \set{0,\dots,a,b,\dots,r}$, a $j$-element subset (green-filled solid oval) of an $r$-element block $M_{t}$ can be chosen in $\ibinom r j$ ways. As there are $v_j$ such subsets and there are several values of $j$, the product in the last row of  \eqref{eq:fprab} is the number how many ways the systems of the green-filled solid ovals can be chosen. Therefore, $\fprab p r a b (n)$ is the number of fundamental pairs as required.

Next, let $(i,B)\neq (i',B')$ be distinct fundamental pairs, $Y=B\cup X\in U(i,B)$, and $Y'=B'\cup X'\in U(i',B')$. For the sake of contradiction, suppose that $Y\subseteq Y'$. If  we had that $i=i'$, then 
$B=(M\setminus M_i)\cap Y \subseteq (M\setminus M_i)\cap Y'= (M\setminus M_{i'})\cap Y'=B'$, which together with $|B|=h=|B'|$ would give that $B=B'$ and so $(i,B)=(i',B')$,  a contradiction. Hence, $i\neq i'$. 
Observe that $Y\subseteq Y'$ gives that $M_j\cap Y\subseteq M_j\cap Y'$ for all $j\in\set{0,\dots, q}$. Furthermore, $M_j\cap Y=B_j$ for $j\neq i$ while  $M_i\cap Y=X$. Similarly, $M_j\cap Y'=B'_j$ for $j\neq i'$ while  $M_{i'}\cap Y'=X'$. Hence, $B_j\subseteq B'_j$ and so   $|B_j|\leq |B'_j|$  for $ j\in\set{0,\dots,q}\setminus\set{i,i'}$, implying that
\begin{equation}
z:=\sum_{ j\in\set{0,\dots,q}\setminus\set{i,i'}} |B_j| \leq \sum_{ j\in\set{0,\dots,q}\setminus\set{i,i'}} |B'_j|=: z'.
\label{eq:kgWsvdshbRmv}
\end{equation}
As  $X$ is medium-sized, $B'_{i}$  is extremal, and $X=M_i\cap Y\subseteq M_i\cap Y'=B'_i$, we have that $B'_i$ is large, that is, $b\leq |B'_i|$.  Hence, \eqref{eq:kgWsvdshbRmv} gives that $z'+b\leq z'+|B'_i|=|B'|$. Similarly, $X'$ is medium-sized, $B_{i'}$ is extremal, and $B_{i'}=M_{i'}\cap Y\subseteq M_{i'}\cap Y'=X'$, whence $B_{i'}$ is small, that is, $|B_{i'}|\leq a$. Thus, $|B|=z+|B_{i'}|\leq z+a$. 
Combining the inequalities $a<b$, $|B| \leq z+a$, $z'+b\leq  |B'|$, and \eqref{eq:kgWsvdshbRmv}, we obtain that
\begin{equation*}
|B| \leq z+a <  z+b \leq z'+b\leq |B'|.
\end{equation*}
This strict inequality contradicts (F1),   completing the proof of Proposition \ref{prop:hYmrbjkSbCklsh}.
 \end{proof}
 
Several ideas and ingredients of the proof above, like the way of  partitioning the base set into blocks, are contained in Dove and Griggs \cite{dovegriggs} and  Katona and Nagy \cite{KatonaNagy}. However, even if the construction given in  \cite{dovegriggs} was tailored to our particular posets $U$, (F1) would fail.
The following assertion says that the lower estimate given in Proposition \ref{prop:hYmrbjkSbCklsh} is  \emph{asymptotically} as good as possible.

 \begin{proposition}\label{prop:snMhrRgntlJ}
For $r\in\Nfromwhat 3$ and  $0\leq a<b\leq r\in \Nplu$ such that $a+2\leq b$, $\frab r a b(n)$ and, for any fixed $p\in\mathbb Z$, $\fprab p r a b(n)$ are asymptotically $\Sp(\FSP r a b,n)$ as $n\to\infty$.
\end{proposition}

\begin{proof} With  $s:=|\FSP r a b|$, 
$s=2^r- \binom{r}{0}-\dots-\binom{r}{a}- \binom{r}{b}-\dots-\binom{r}{r}$. Let $\kappa$ be a real number such that $\kappa< 1$ but $1-\kappa$ is very little. 
As we have that $\sum_{i=0}^\infty ((2^r-s)/2^r)^i = 2^r/s$, we can pick an $n_0\in\Nplu$ such that 
\begin{equation}
\kappa\cdot 2^r/s \leq \sum_{i=0}^{\lint{n/r}-1}  ((2^r- s)/2^r)^i \leq \frac 1\kappa  2^r/s\text{ for all }n\text{ such that }n\geq n_0.
\label{eq:kmsgtsvznlkflMrjK}
\end{equation}
It suffices to deal with $\fprab p r a b$ for a fixed  $p\in\mathbb Z$. Using \eqref{eq:vZvsznRzbnLjdnkpst}, we can pick an $n_1\geq n_0$ such that  
\begin{equation}
\begin{aligned}
\kappa &\cdot\fsp(n)\cdot 2^{-(i+1)r} \cr
&\leq \binom{n-(i+1)r}{p+\lint{(n-r)/2}-0v_0-1v_1-\dots -av_a-bv_b-\dots - rv_r} \cr
&\leq  \frac 1\kappa \cdot  \fsp(n)\cdot 2^{-(i+1)r}
\end{aligned}
\label{eq:sblvnHgnKrsTbk}
\end{equation}
for all $n\geq n_1$.  Let us define an auxiliary function for $n\geq n_1$ and apply the multinomial theorem to it as follows.
  \begin {align}
 \faux&(n):=\sum_{i=0}^{\lint{n/r}-1} \sum_{{\sct{\vec v\in\set{0,\dots,i}^{r+a-b+2}}\atop{\sct{v_0+\dots+v_a+v_b+ \dots+v_{r}=i}}}}
  \frac{i!}{v_0!\dots v_a! \cdot v_b!\dots  v_r!}{\,\,}\times \cr
 &\times
 \fsp(n)\cdot 2^{-(i+1)r} \binom{r}{0}^{\kern -3pt v_0} \dots \binom{r}{a}^{\kern -3pt v_a}
 \cdot  \binom{r}{b}^{\kern -3pt v_b}\dots \binom{r}{r}^{\kern -3pt v_r}
\label{eq:nphszprcvMNsTs}\\
 &=  \frac{\fsp(n)}{2^r}\sum_{i=0}^{\lint{n/r}-1} (2^{r})^{-i}  \Bigl(  \binom{r}{0}+\dots+\binom{r}{a}+ \binom{r}{b}+\dots+\binom{r}{r} \Bigr)^i
 \cr
  &=  \frac{\fsp(n)}{2^r}\sum_{i=0}^{\lint{n/r}-1}  \Bigl(\frac {2^r-s}{2^r}\Bigr)^i.
  \label{eq:smrszksWbntnkPnk}
\end{align}
Comparing \eqref{eq:fprab}, \eqref{eq:sblvnHgnKrsTbk}, and \eqref{eq:nphszprcvMNsTs}, we obtain that  $\kappa \faux(n)  \leq \fprab p r a b(n)    \leq     \kappa^{-1} \faux(n)$ holds for all  $n\geq n_1$. Applying \eqref{eq:kmsgtsvznlkflMrjK} to the sum in  \eqref{eq:smrszksWbntnkPnk}, it follows that  $\kappa \fsp(n)/s \leq \faux(n) \leq \frac 1\kappa \fsp(n)/s $. Substituting this pair of inequalities into the previous one, we have that $\kappa^2\fsp(n)/s\leq \fprab p r a b(n) \leq \kappa^{-2}\fsp(n)/s$ for all $n\geq n_0$.  Letting $\kappa\to 1$, it follows that $\fprab p r a b(n)$ is asymptotically $\fsp(n)/s$. So is $\Sp(\FSP r a b,n)$  by Dove and Griggs \cite{dovegriggs} and Katona and Nagy \cite{KatonaNagy}. By transitivity, we obtain the required asymptotic equality.
 The proof of Proposition \ref{prop:snMhrRgntlJ} is complete.
\end{proof}

\section{Pairs of estimates}
For $n\in\Nfromthree$, take the following ``discrete $4$-dimensional simplex''\begin{equation}
\begin{aligned}
H_4(n):=\set{(t,x_1,x_2,x_3)\in \Nnul^4: x_1>0,\,\,  x_2>0,\,\,  x_3>0,\,\,\cr
 t+x_1+x_2+x_3\leq n}.
\end{aligned}
\label{eq:cskzgSzbnzbW}
\end{equation} 
Remembering that $[3]:=\set{1,2,3}$, define the function $\fha\colon H_4(n)\to \Nnul$ by
\begin{equation}
\begin{aligned}
\fha(t,x_1,x_2,x_3)&=\sum_{j\in[3]} (t+x_j)!\cdot(n-t-x_j)!\cr
&+\sum_{\set{j,u}\subseteq[3],\,\, j\neq u} (t+x_j+x_u)!\cdot(n-t-x_j-x_u)!\cr
&- \sum_{(j,u)\in[3]\times [3],\,\, j\neq u} (t+x_j)!\cdot x_u!\cdot(n-t-x_j-x_u)!\,,
\end{aligned}
\label{eq:hlnCsvpFd}
\end{equation}
and let 
\begin{equation}
M_n:=\min\set{\fha(t,x_1,x_2,x_3):  (t,x_1,x_2,x_3)\in H_4(n)}.
\label{eq:prbrCJLqsLpd} 
\end{equation}
We also define the following three functions:
\allowdisplaybreaks{
\begin{align}
&\fgr r(n):=\left\lfloor \frac 1 2\fsp(n+2-r)  \right\rfloor,  \label{eq:thfGvb}\\
&\begin{aligned}
\fgha(n):=\lint{n!/M_n}, \text{ where }M_n\text{ is given in }\eqref{eq:prbrCJLqsLpd}, \text{ and}
\end{aligned}
 \label{eq:thfGvc}\\
 &\begin{aligned}
\fghb(n)=\Bigl\lfloor n! \cdot
\Bigl(3\cdot \lint{n/2}!\cdot\uint{n/2}! +
3\cdot \lint{(n+2)/2}!\cdot\uint{(n-2)/2}! \cr
- 6\cdot \lint{n/2}!\cdot \uint{(n-2)/2}!
\Bigr)^{-1} \Bigr\rfloor 
\end{aligned}
\label{eq:mclSszrR}
\end{align}
}%

Next, based on the notations and concepts given in \eqref{eq:Spnotat}, \eqref{eq:fwftspRtd}, Definition \ref{def:flp},   \eqref{eq:thfGvb}, \eqref{eq:thfGvc}, and \eqref{eq:mclSszrR}, we can formulate the main result of the paper.

\begin{theorem}\label{thm:main}
For $3\leq r\leq n\in\Nplu$ and $p\in\set{-r,-r+1,\dots, r-1,r}$,
 $\fgr r(n)$ is an upper  estimate while
\begin{align}
\fprab p r 0 r(n)&:=\sum_{i=0}^{\lint{n/r}-1} \sum_{j=0}^i \binom i j 
\binom{n-(i+1)r}{p+\lint{(n-r)/2}-jr} \,\,\text{ and}
\label{eq:fckZls}
\\
\flatrab r 0 r (n)&:=%\max\set{ \fprab p r 0 r(n): p\in\set{-r,-r+1,\dots,r-1,r}}
   \fprab 0 r 0 r(n) 
\label{eq:fmkVlphRg}
\end{align}
are lower estimates of $\Sp(\FSP r 0 r,n) = \Sp(\Jir{\FD r},n)$ on $\Nfromwhat r$. In particular, 
\begin{equation}
\text{for all }n\in \Nfromwhat r,\quad \flatrab  r 0 r(n) \leq  \Sp(\Jir{\FD r},n) \leq \fgr r(n).
\label{eq:vhbnkGlmBk}
\end{equation}
For $r=3$, in addition to the satisfaction of \eqref{eq:vhbnkGlmBk}, $\fgha(n)$ is also an upper estimate of $\Sp(\Jir{\FD 3},n)$ on $\Nfromwhat 3$. For $n\in\set{3,4,\dots,300}$,  $\fgha(n)=\fghb(n)\leq \fgr r(n)$; in fact,  $\fghb(n) <  \fgr r(n)$ for $n\in\set{5,6,\dots,300}$. The pair $( \flatrab  3 0 3,  \fgr 3 )$ is separated for $n\in\Nfromwhat 3$, and so are the pairs $( \flatrab  3 0 3,  \fghb)$  and $( \flatrab  3 0 3,  \fgha)$   for  $n\in\set{3,4,\dots,300}$. Finally, for $r\in\set{3,4,\dots,100}$, the pair $(\flatrab  r 0 r, \fgr r)$ is separated on the set $\set{r,r+1,\dots, 300}$.
\end{theorem}

Our computer-assisted computation, which took 952 seconds $\approx$ 16 minutes, shows that for $r\in\set{3,\dots,200}$ and $n\in  \set{r,\dots 300}$, $\flatrab r 0 r(n)$ defined in \eqref{eq:fmkVlphRg} 
is the same as $\frab r 0 r(n)$ ; see \eqref{eq:frab}. Since $\flatrab r 0 r(n)$ is easier to define and much easier to compute than $\frab r 0 r(n)$, it is the former that occurs in  Theorem \ref{thm:main}. However, it will be clear from its proof that the theorem holds with $\frab r 0 r$ in place of $\flatrab r 0 r$.

\begin{conjecture}\label{conj:nNrs}
We guess that   $\fgha(n) = \fghb(n)$ for all $n\in\Nfromwhat 3$ and  $\fghb(n) <  \fgr r(n)$ for all $\Nfromwhat 6$.
\end{conjecture}

Example \ref{example:ABCDE} in Section \ref{sect:odd} will show that,  combining Theorem \ref{thm:main} with  Observation \ref{obs:nthNmb}, we can determine   $\gmin{\FD 3^k}$ exactly in many cases and we can give a good approximation for 
$\gmin{\FD r^k}$ quite often.

\begin {proof}[Proof of Theorem \ref{thm:main}]
Substituting $(i-j,j)$ for $(v_0,v_r)$ and observing that the multinomial coefficient becomes a binomial one, it is clear that $\fprab p r 0 r$ in \eqref{eq:fckZls} is a particular case of \eqref{eq:fprab}.  Hence, Lemma \ref{lemma:wDwdRjRd}, \eqref{eq:frab}, Proposition   \ref{prop:hYmrbjkSbCklsh},  and  \eqref{eq:fmkVlphRg} yield the first inequality in \eqref{eq:vhbnkGlmBk}. 

By its definition (and Lemma \ref{lemma:wDwdRjRd}), $\FSP r 0 r=\Jir{\FD r}\cong \ntPow{[r]}$. 
In each of the intervals $[\set{1}, \set{1,3,4,\dots,r}]$ and $[\set{2}, \set{2,3,4,\dots,r}]$, take a maximal chain; denote these two chains by $C'$ and $C''$. Clearly, $C'$ and $C''$ are unrelated chains of length $r-2$ and $C'\cong C''$. Let $n\in\Nfromwhat r$. With $k:=\Sp(\FSP r 0 r, n)$, there are $k$ pairwise unrelated copies of $\ntPow{[r]}\cong \FSP r 0 r$ in $\nPow$.  Therefore, there are $2k$ pairwise unrelated copies of $C'$ in $\nPow$. So $2k\leq \Sp(C',n)$. By Griggs, Stahl, and Trotter \cite{griggsatall}, $\Sp(C',n)=\fsp(n-(r-2))$. So $2k\leq \fsp(n+2-r)$, implying the second inequality in \eqref{eq:vhbnkGlmBk}.

In the rest of the proof, $r:=3$.  Let  $\Sym n$  stand for the set of all permutations of $[n]$.  For $\vs=(\sigma_1,\dots,\sigma_n)\in\Sym n$
and $i\in\set{0,1\dots,n}$, the $i$'s \emph{\acr initial \acr segment} of $\vs$ is $\inseg\vs  i:=\set{\sigma_j:j\leq i}$. For $X\in \nPow$, the \emph{\acr permutation \acr set} associated with $X$ is 
$\pset X:=\set{\vs\in\Sym n: X=\inseg\vs{|X|}}$.
The trivial fact that
\begin{equation}
\parbox{7cm}{ if $X,Y\in\nPow$ are incomparable (in notation, $X\parallel Y$), then $\pset X\cap\pset Y=\emptyset$}
\label{eq:sBrhCnGs}
\end{equation}
 was used first by  Lubell \cite{lubell}, and then by  Griggs, Stahl, and Trotter \cite{griggsatall}  and some other papers listed in the bibliographic section. To ease the notation, let $\crn:=\FSP 3 0 3$ and denote its elements by $A$, $B$, $C$, $X$, $Y$, $Z$ according to Figure \ref{figone}. 
Let $k:=\Sp(\crn,n)$, and 
let $\icrn 1,\dots,\icrn k$ be pairwise unrelated copies of $\crn$ in $\nPow$. For $\icrn i$, we use the notation $\icrn i=\set{A_i,B_i,C_i,X_i,Y_i,Z_i}$  in harmony with Figure \ref{figone}; for example, $A_i\subset X_i$ and $A_i\parallel Z_i$, etc.. We claim that $\icrn 1, \dots, \icrn k$ can be chosen so that, for all $i\in[k]$,
\begin{align}
X_i=A_i\cup B_i,\quad Y_i=A_i\cup C_i, \quad Z_i=B_i\cup C_i,
\label{eq:bRslkMna}
\\
A_i=X_i\cap Y_i,\quad B_i=X_i\cap Z_i,\quad C_i=Y_i\cap Z_i.
\label{eq:bRslkMnb}
\end{align}
Assume that the first equality in \eqref{eq:bRslkMna} fails. Let $X_i':=A_i\cup B_i$ and define $\icrn i{}':=(\icrn  i\setminus\set{X_i})\cup\set {X_i'}$.
If we had that $X_i'\subseteq Y_i$, then $B\subseteq X_i'\subseteq Y_i$ would be a contradiction. As $Y_i\subseteq X_i'$ would lead to $Y_i\subseteq X_i$ since $X_i'\subseteq X_i$, we conclude that $X_i'\parallel Y_i$. We obtain similarly that $X_i' \parallel Z_i$. So $\set{X_i',Y_i,Z_i}$ is an antichain, and now it follows easily that  $\icrn i{}'$ is a copy of $\crn$. For $j\in[k]\setminus\set i$ and $E\in \icrn j$, $E\subseteq X_i'$ would lead to $E\subseteq X_i$ while $X_i'\subseteq E$ to $A_i\subseteq E$. So $E\nparallel X_i'$ would lead to contradiction. Hence, $\icrn i{}'$ and $\icrn j$ are unrelated, showing that we can change $\icrn i$ to $\icrn i{}'$.  As there is an analogous treatment for $Y_i$ and $Z_i$, and we can take $i=1$, $i=2$, \dots, $i=k$ one by one,  \eqref{eq:bRslkMna} can be assumed. 

Recall that  Gr\"atzer \cite[Lemma 73]{GLTFound} asserts that whenever $a,b,c$ are elements of a lattice such that $\set{a\vee b, a\vee c, b\vee c}$ is a 3-element antichain, then this antichain generates an 8-element Boolean sublattice in which  $\set{a\vee b, a\vee c, b\vee c}$ is the set of coatoms. Therefore,  if we apply the dual of the procedure above (that is, if we replace $A_i$ by $X_i\cap Y_i$, etc.), then we reach \eqref{eq:bRslkMnb} without destroying the validity of \eqref{eq:bRslkMna}. We have shown that both \eqref{eq:bRslkMna} and \eqref{eq:bRslkMnb} can be assumed; so we assume them in the rest of the proof.

Let $T_i:=X_i\cap Y_i\cap Z_i$. By \eqref{eq:bRslkMnb}, $T_i$ is also the intersection of any two of $A_i$, $B_i$, and $C_i$. Hence, letting $\bul{A_i}:=A_i\setminus T_i$, $\bul{B_i}:=B_i\setminus T_i$,  and  $\bul{C_i}:=C_i\setminus T_i$,  it follows from \eqref{eq:bRslkMna}, \eqref{eq:bRslkMnb}, and  $\icrn i\cong \crn$ that $\bul{A_i}$, $\bul{B_i}$, and $\bul{C_i}$  are   pairwise disjoint subsets of $[n]$, none of them is empty, they are  disjoint from $T_i$, and 
\begin{equation}
\begin{aligned}
A_i&=T_i\cup\bul{A_i}, \quad B_i=T_i\cup\bul{B_i}, \quad C_i=T_i\cup\bul{C_i}, \cr
X_i&=T_i\cup\bul{A_i}\cup\bul{B_i},\quad
Y_i=T_i\cup\bul{A_i}\cup\bul{C_i},\quad
Z_i=T_i\cup\bul{B_i}\cup\bul{C_i}.
\end{aligned}
\label{eq:rcnDwrT}
\end{equation}
For $i\in [k]$, we let
\begin{equation}
G_i:=\pset{A_i}\cup\pset{B_i}\cup\pset{C_i}\cup\pset{X_i}\cup\pset{Y_i}\cup\pset{Z_i}.
\label{eq:frnCpTknTlj}
\end{equation}
As each of $A_i$,\dots,$Z_i$ is incomparable with each of $A_j$,\dots,$Z_j$ provided that $i\neq j$, \eqref{eq:sBrhCnGs} together with \eqref{eq:frnCpTknTlj} imply that 
\begin{equation}
\text{for }i,j\in[k],\
\text{ if }\,i\neq j\,\text{ then }\, G_i\cap G_j=\emptyset.
\label{eq:rmKhmpRdc}
\end{equation}
It follows from \eqref{eq:rmKhmpRdc},  $G_1\cup\dots\cup G_k\subseteq \Sym n$, and $|\Sym n|=n!$ that  
\begin{equation}
\sum_{i\in [k]}|G_i|\leq n!\,\,.
\label{eq:nkVnkzmPszjNl}
\end{equation}

Next, for $i\in[k]$, we focus on $|G_i|$. Denote $|T_i|$, $|\bul{A_i}|$, $|\bul{B_i}|$, and  $|\bul{C_i}|$ by  $t_i$,   $a_i$, $b_i$, and $c_i$, respectively. By \eqref{eq:rcnDwrT}, $|A_i|=t_i+a_i$, $|B_i|=t_i+b_i$, $|C_i|=t_i+c_i$, $|X_i|=t_i+a_i+b_i$,  $|Y_i|=t_i+a_i+c_i$, and $|Z_i|=t_i+b_i+c_i$. Observe that  $|\pset{A_i}|=(t_i+a_i)!\cdot (n-t_i-a_i)!$ since what the first $|A_i|=t_i+a_i$ components of $\vs=(\sigma_1,\dots,\sigma_n)\in \pset{A_i}$ form is set $A_i$ and they can be arranged in $(t_i+a_i)!$ many ways while the rest of the components of $\vs$ in the last $n-t_i-a_i$ positions in $(n-t_i-a_i)!$ many ways. We obtain similarly that   $|\pset{B_i}|=(t_i+b_i)!\cdot (n-t_i-b_i)!$,  $|\pset{C_i}|=(t_i+c_i)!\cdot (n-t_i-c_i)!$,  $|\pset{X_i}|=(t_i+a_i+b_i)!\cdot (n-t_i-a_i-b_i)!$, $|\pset{Y_i}|=(t_i+a_i+c_i)!\cdot (n-t_i-a_i-c_i)!$, and $|\pset{Z_i}|=(t_i+b_i+c_i)!\cdot (n-t_i-b_i-c_i)!$. 
%It is not sufficient to take the sum of these numbers to obtain $|G_i|$ as 
%The sets of permutations we have just considered are not pairwise disjoint. However, 
It follows from \eqref{eq:sBrhCnGs} that the intersection of any three of the six permutation  sets considered above is empty since there is no 3-element chain in $\icrn i$. By  \eqref{eq:sBrhCnGs} again, we need to take care of the intersections of two permutation sets associated with comparable members of $\icrn i$; there are six such intersections as the diagram of $\crn$ has exactly six edges; see Figure \ref {figone}. One of the just-mentioned six intersections is $\pset{A_i}\cap\pset{X_i}$. For a permutation $\vs\in \pset{A_i}\cap \pset{X_i}$, \eqref{eq:rcnDwrT} yields that there are $|A_i|!=(t_i+a_i)!$ possibilities to arrange the elements of $A_i$ in the first $|A_i|$ places,  $b_i!$ many possibilities to arrange the elements of $X_i\setminus A_i=\bul{B_i}$ in the next $b_i$ places, and $(n-t_i-a_i-b_i)!$ possibilities for the rest of entries of $\vs$. Hence, $|\pset{A_i}\cap\pset{X_i}|=(t_i+a_i)!\cdot b_i!\cdot (n-t_i-a_i-b_i)!$, and analogously for the other five intersections of two permutation sets.

The considerations above imply that for $i\in [k]$, $|G_i|=\fha(t_i,a_i,b_i,c_i)$; see \eqref{eq:hlnCsvpFd}.    As $(t_i,a_i,b_i,c_i)$ is clearly in $H_4(n)$,  \eqref{eq:prbrCJLqsLpd}  yields that $M_n\leq |G_i|$.  This fact and \eqref{eq:nkVnkzmPszjNl}  imply that $kM_n\leq \sum_{i\in[k]}|G_i| \leq n!$. Dividing  by $M_n$ and taking into account that $k\in\Nplu$, we obtain that $\Sp(\crn,n)=k\leq \lint{n!/M_n}= \fgha(n)$, as required.

We only guess but could not prove that for all $n\in\Nfromthree$,  $\fha$ takes its minimum on $H_4(n)$ at $(\lint{(n-2)/2},1,1,1)$; see also  Conjecture \ref{conj:nNrs}. However, we can reduce the computational difficulties by considering the following auxiliary function:
\begin{equation}
\begin{aligned}
\fhb &(t,x,y)=(t+x)!\cdot(n-t-x)! + (t+y)!\cdot(n-t-y)!\cr
&+2(t+x+y)!\cdot(n-t-x-y)!  -2(t+x)!\cdot y!\cdot(n-t-x-y)! \cr
& -2(t+y)!\cdot x!\cdot(n-t-x-y)!  \,\, .
\end{aligned}
\label{eq:mrNskDkfhb}
\end{equation}
The definition of $H_4(n)$, see \eqref{eq:cskzgSzbnzbW}, and 
\begin{equation}
2\fha(t,x_1,x_2,x_3)=\fhb (t,x_1,x_2) + \fhb (t,x_2,x_3) + \fhb (t,x_1,x_3)\, ,
\label{eq:ntsbGkwCFJvbtr}
\end{equation}
explain that we are interested in $\fhb $ on the first one of the following two sets,
\begin{align}
H_3(n)&:=\set{(t,x,y)\in \Nnul^3: x>0,\,\,  y>0,\,\,  
 t+x+y \leq n-1}\text{ and}
 \label{eq:t-hedron}\\
H_3'(n)&:=\set{(t,x,y)\in \Nnul^3: x>0,\,\,  y\geq x,\,\,  
 t+x+y \leq n-1}.
 \label{eq:rtmShrdrnb} 
\end{align}
In \eqref{eq:t-hedron},  the sum is only at most $n-1$ since the fourth variable of $\fha $, which does not occur in $\auxf$, is at least 1. The progress is that  $H_3(n)$ has much less elements than $H_4(n)$, and $H'_3(n)$ has even less; this is why we could reach  $300$ in Theorem \ref{thm:main}. (Note that a priori, it was not clear that when $2\fha (t,x_1,x_2,x_3)$ takes its minimum value, then so do all of its summands in \eqref{eq:ntsbGkwCFJvbtr}.)  Observe that since $\fhb $ is symmetric in its last two variables,
 \begin{equation}
 \min\set{\fhb (t,x,y):(t,x,y)\in H_3(n)}= \min\set{\fhb (t,x,y):(t,x,y)\in H'_3(n)}.
\label{eq:mRprckCRjt}
\end{equation}
A straightforward (computer algebraic) Maple program, which benefits from \eqref{eq:mRprckCRjt}, shows that
\begin{equation}
\parbox{8.6cm}{for $3\leq n\leq 300$, $\fhb $ takes its minimum  on the discrete tetrahedron $H_3(n)$ at $(t,x,y)=(\lint{(n-2)/2},1,1)$.}
\label{eq:pbxfJbcSplt}
\end{equation}
 (Note that $\auxf$ takes its minimum at two triples if $n$ is even  but only at a unique triple if $n$ is odd.)   For more about the program, see the (Appendix) Section \ref{sect:maple} in the extended 
 \href{https://arxiv.org/abs/2309.13783}{arXiv:2309.13783}  (or  \href{https://arxiv.org/abs/2309.13783v2}{arXiv:2309.13783v2})   version% \red{xxx}
 \footnote{\emph{This version} that we are reading is the extended \href{https://arxiv.org/abs/2309.13783}{arXiv:2309.13783} version.} 
 of the paper; at the time of writing, the program is also available from the 
 \href{http://www.math.u-szeged.hu/~czedli}{author's website}. Here we only mention that the computation for \eqref{eq:pbxfJbcSplt} took a bit more than seven hours.

If  $n\in\set{3,4,\dots,300}$ and  $(\lint{(n-2)/2},1,1,1)$ is substituted for $(t,x,y,z)$, then  each of the three summands in \eqref{eq:ntsbGkwCFJvbtr} takes its minimal value by \eqref{eq:pbxfJbcSplt}. This allows us to conclude that  at $(t,x,y,z)=(\lint{(n-2)/2},1,1,1)$, $\fha$ takes its minimum on $H_4(n)$.  Thus, for $n\in\set{3,4,\dots,300}$  and for $M_n$ defined in \eqref{eq:prbrCJLqsLpd},
\begin{equation}
\begin{aligned}
M_n&=\fha (\lint{(n-2)/2},1,1,1)=3\cdot \lint{n/2}!\cdot\uint{n/2}!  \cr
&+
3\cdot \lint{(n+2)/2}!\cdot\uint{(n-2)/2}!  
- 6\cdot \lint{n/2}!\cdot \uint{(n-2)/2}!\,\, .
\end{aligned}
\label{eq:prNsJsLsd} 
\end{equation}
Combining \eqref{eq:thfGvc}, \eqref{eq:prNsJsLsd}, and \eqref{eq:mclSszrR}, we obtain that $\fgha(n)=\fghb(n)$  for  $n$ belonging to the set $\set{3,4,\dots,300}$, as required.

Next,  to show that the pair $(\flatrab  3 0 3,\fgr 3)=(\fprab 0  3 0 3,\fgr3)$ is separating, we need to show that $\fprab 0 3 0 3(n+1)-\fgr3(n)\geq 0$ for all $n\in \Nfromthree$. Depending on the parity of $n$, there are two cases. If $n$ is of the form $n=2m+2$ then, reducing the sum in \eqref{eq:fckZls} to its summands corresponding to $(i,j)=(0,0)$ and $(i,j)=(1,0)$,  %{xxx ellenorizve:}
\begin{gather}
2\fprab 0 3 0 3(n+1)-2\fgr3(n)\geq 2\binom{2m}{m}+2\binom{2m-3}{m} - \binom{2m+1}{m}
\label{eq:szlfRmrkfld}\\
=\frac{2\cdot(2m)!}{m!\cdot m!} + \frac{2\cdot(2m-3)!}{m!(m-3)!} -\frac{(2m+1)!}{m!(m+1)!}\cr
=\frac{(2m-3)!}{m!(m+1)!}\cdot \alpha,\quad\text{ where}\quad \alpha=2(m+1)2m(2m-1)(2m-2)\cr
+2(m+1)m(m-1)(m-2)  - (2m+1)2m(2m-1)(2m-2)
\cr
=2m^4+4m^3-14m^2+8m=2m(m+4)(m-1)^2.
\end{gather}
Hence, both $\alpha$ and the fraction multiplied by $\alpha$ are non-negative for $ m\in\Nplu$. Thus, $\fprab 0 3 0 3(n+1)-\fgr3(n)\geq 0$  for $n\geq 4$ even.
Similarly, for $n=2m+1$ odd,  %{xxx ellenorizve:}
\begin{gather*}
2\fprab 0 3 0 3(n+1)-2\fgr3(n)\geq 
2\binom{2m-1}{m-1}+2\binom{2m-4}{m-1} - \binom{2m}{m}\cr
=\frac{(2m-4)!}{m!m!}\cdot  2m^2(m-1)(m-2).
\end{gather*}
Therefore, $\fprab 0 3 0 3(n+1)-\fgr3(n)\geq 0$ for  $2\leq m\in\Nplu$, that is, for $n\geq 5$ odd. For $n=3$,  $\fprab 0 3 0 3(n+1)-\fgr3(n)\geq 0$ is trivial; see also \ref{eq:szbklhrm}. 
We have shown that $(\flatrab 3 0 3,\fgr3)$ is separated.

The already mentioned Maple program  has computed $\fgr3(n)$, $\fgha(n)$, and $\fghb(n)$ for all $n\in\set{3,4,\dots,300}$. This computation proves that $\fghb(n)= \fgha(n)\leq \fgr3(n)$ for all these $n$ and $\fghb(n) =\fgha(n)<\fgr3(n)$ for $n\in\set{5,6,\dots,300}$. These inequalities and that $(\flatrab 3 0 3,\fgr3)$ is separated imply that $(\flatrab 3 0 3,\fgha)$ and  $(\flatrab 3 0 3,\fghb)$ are separated on $\set{3,4,\dots,300}$. The same Maple program has computed all the relevant $\flatrab r 0 r(n+1)$ and $\fgr r(n)$, from which we conclude that for $r\in\set{3,4,\dots,100}$, the pair $(\flatrab r 0 r, \fgr r)$ is separated on the set $\set{r,r+1,\dots, 300}$. The proof of Theorem \ref{thm:main} is complete.
\end{proof}

Some comments on this proof are appropriate here. While we could use quite a rough estimation in \eqref{eq:szlfRmrkfld} when proving that $(\flatrab 3 0 3,\fgr 3)$ is separating on the set $\Nfromwhat 3$, there is no similar possibility for  $(\flatrab r 0 r,\fgr r)$. Indeed, since  $\flatrab r 0 r(n+1)=\fgr r(n)$ for,  say, $(r,n)=(20,56)$ when $\flatrab {20}0{20}(56+1)=17\,672\,631\,900=\fgr{20}(56)$,  no estimation would be possible.  As $\fgr r(n)$ is far from being asymptotically good, it is not worth putting more work into its investigation. While  we could use  Gr\"atzer \cite[Lemma 73]{GLTFound}  to reach a pleasant situation for $r=3$, see \eqref{eq:bRslkMna} and \eqref{eq:bRslkMnb},we have no similar tool for $r>3$; this explains that Theorem \ref{thm:main} does not tell too much about upper estimates in case of $r>3$. Finally, note that  even though $\fhb$ in \eqref{eq:mrNskDkfhb} is simpler than $\fha$  in \eqref{eq:hlnCsvpFd}, the three-variate function $\fhb$ is still too complicated. In particular, we know from computer-assisted calculations that $\fhb$ has several ``local minima'' on the discrete tetrahedron $H_3(n)$ defined in  \eqref{eq:t-hedron}; this is our excuse that we could verify Conjecture \ref{conj:nNrs} only for $n\leq 300$ and only with a computer.

\section{Odds and ends, including some computational results}
\label{sect:odd}
Theorem \ref{thm:main} pays no attention to the case  $r=2$, which is trivial by the following remark.  As in \eqref{eq:thfGvb},   $\fgr 2(n):=\lint{\fsp(n)/2}=\lint{\ibinom{n}{n/2}/2}$.

\begin{remark} \label{rem:mlGjKtknvZp}
For $n\in\Nfromwhat 2$,   $\Sp(\Jir{\FD 2},n)=\fgr 2(n)$.
\end{remark}

\begin{proof} By Lemma  \ref{lemma:wDwdRjRd} or trivially, $\Jir{\FD 2}$ is the two-element antichain.  Hence, Remark \ref{rem:mlGjKtknvZp} follows from  Sperner's theorem; see \eqref{eq:fspnotat}. 
\end{proof}

\begin{corollary}\label{corol:kHhsBd}
 For  $ r\in\Nfromwhat 3$ and $k\in\Nfromwhat 2$,   let $n\in\Nplu$ be the smallest integer such that $k\leq \flatrab r 0 r(n)$; see \eqref{eq:fmkVlphRg}. Then for every distributive lattice $D$ generated by $r$ elements,  the direct power $D^k$ has an at most $n$-element generating set.
\end{corollary}

\begin{proof}
Let $k$, $D$, and $n$ be as in the corollary. Since $k\leq \flatrab r 0 r(n)$ is included in the assumption and $ \flatrab r 0 r(n)\leq  \Sp(\Jir{\FD r},n)$ by Theorem \ref{thm:main},  it follows from \eqref{eq:zkRspHjTKc} that  $\FD r^k$ can be generated by an at most $n$ element subset $Y$.  Using that $\FD r$ is the \emph{free} $r$-generated distributive lattice, we can pick a surjective (in other words, onto) homomorphism $\phi\colon\FD r\to D$. Then $\phi^k\colon \FD r^k\to D^k$, defined by $(x_1,\dots,x_k)\mapsto(\phi(x_1),\dots,\phi(x_k))$, is also a surjective homomorphism. Thus, 
$\phi^k(Y)$ generates $D^k$ and  $|\phi^k(Y)|\leq |Y|\leq n$ proves  Corollary \ref{corol:kHhsBd}.
\end{proof}

The just-proved corollary and the abundance of large lattices that are easy-to-describe and easy-to-work-with motivate the following extension of the cryptographic ``protocol'' outlined in Cz\'edli \cite{czgDEBRauth} and, mainly, in \cite{czgboolegen}.  The purpose of the  quotient marks here is \emph{to warn the reader}: none of our protocols is fully elaborated and, thus, it does not meet the requirements of  nowadays' cryptology. In particular, neither a concrete method of choosing the master key according to some probabilistic distribution is given nor we have proved that the average case withstands attacks; we do not even say that we are close to meet these requirements.  On the other hand,  no rigorous average case analysis supports some widely used and, according to experience, safe cryptographic protocols like RSA and AES and, furthermore, many others rely ultimately on the \emph{conjecture} that the complexity class \tbf{P} is different from \tbf{NP}.  This is our excuse to tell a bit more about one of our motivations in Remark \ref{rem:kCpgCl} below.  
 For a lattice $L$ and $\vec h=(h_1,\dots,h_k)\in L^k$,   $\vec h$ is a ($k$-dimensional) \emph{generating vector} of $L$ if $\set{h_1,\dots, h_k}$ is a generating set of $L$. 

\begin{remark}\label{rem:kCpgCl} In the \emph{session key exchange protocol} given in Cz\'edli \cite{czgboolegen}\footnote{At the time of writing, the direct link is
\href{https://arxiv.org/abs/2303.10790v3}{https://arxiv.org/abs/2303.10790v3}, see (4.3) in version 3 of the paper.}, the secret master key known only by the communicating parties was a $k$-dimensional generating vector $\vec h$ of the $2^n$-element Boolean lattice $B_n$.  The point was that $\gmin{B_n}$ is small, and so there are very many $k$-dimensional generating vectors $\vec h$ if $k$ is a few times, say, seven  times larger than  $\gmin{B_n}$. Here we suggest to add \textup{(A)} or \textup{(B)} to the protocol outlined in \cite{czgboolegen} and to work in a lattice different from $B_n$.

\textup{(A)} Choose a medium-sized finite random poset $U$ and an exponent $n\in\Nplu$; for example, a 20-element random poset $U$ and $n=500$ are sufficient. (There are very many 20-element posets; see A000112  in Sloan \cite{sloan}; the direct link is \href{https://oeis.org/A000112}{https://oeis.org/A000112}.) By the well-known structure theorem of finite distributive lattices, see  Gr\"{a}tzer \cite[Theorem 107]{GLTFound}, $U$ determines a finite distributive lattice $D$. Then replace $B_n$ with $D^n$ in the \cite{czgboolegen}-protocol so that, in addition to $\vec h$, $U$ and $n$  also belong to the secret master key.

\textup{(B)} Choose a  random poset  $U$ of size 100 or so.  As in \cite{czgqoufilt}, this $U$ determines the huge lattice $(\textup{Quo}^{\leq}(U);\subseteq)$ of quasiorders extending $\leq_U$; this lattice  can be generated by few elements. Use this lattice instead of $B_n$. The poset  $U$ and a $k$-dimensional generating vector of $(\textup{Quo}^\leq(U);\subseteq)$  constitute the secret master key; otherwise the protocol is the same as in \cite{czgboolegen}.
\end{remark}

Next, we present some computational data;  Section \ref{sect:maple} will explain how these data were obtained by using a computer; at the ``$\approx$'' rows, the last decimals are correctly rounded. 
%red{xxx checked: \eqref{eq:szbklhrm}:4, \eqref{eq:szbklngy}3,  \eqref{eq:szbklot}3, \eqref{eq:cSmrWlKmlG}2, \eqref{eq:cmjsPjts}2}
%
\begin{equation}
\begin{tabular}{l|r|r|r|r|r|r|r|r|r|r}
$n=$ &  3& 4& 5& 6& 7& 8  \cr
\hline
$\flatrab 3 0 3(n)$ & 1 & 1& 2& 3& 6& 11\cr
\hline
$\fgha(n)=\fghb(n)$ & 1 & 1& 2 & 4& 7& 13\cr
\hline
$\fgr 3(n)$ & 1 & 1& 3 & 5& 10& 17\cr
\hline\hline
$n=$ &  9& 10& 11& 12& 13& 14   \cr
\hline
$\flatrab 3 0 3(n)$ & $24$ & $42$& $84$ & $153$& $306$& $570$ \cr
\hline
$\fgha(n)=\fghb(n)$ &$26$& $46$& $92$ & $168$& $333$& $616$ \cr
\hline
$\fgr 3(n)$ &$35$& $63$& $126$ & $231$& $462$& $858$ \cr
\hline\hline
$n=$ &  15& 16& 17& 18& 19& 20  \cr
\hline
$\flatrab 3 0 3(n)$ & 1146 & 2145& 4290& 8100& 16200& 30786\cr
\hline
$\fgha(n)=\fghb(n)$ & 1225 & 2288& 4558 & 8580& 17107& 32413\cr
\hline
$\fgr 3(n)$ & 1716 & 3217& 6435 & 12155& 24310& 46189\cr
\hline
\end{tabular} 
\label{eq:szbklhrm}
\end{equation}
\begin{equation}
\begin{tabular}{l|r|r|r|r|r|r|r|r|r|r}
$n=$ &  4& 5& 6& 7& 8& 9&10&11&12  \cr
\hline
$\flatrab 4 0 4(n)$ & 1 & 1& 2 & 3& 6& 10& 20& 36& 74\cr
\hline
$\fgr 4(n)          $ & 1 & 1& 3 & 5&10& 17& 35& 63& 126\cr
\hline\hline
$n=$ &  13& 14& 15& 16& 17& 18&19&20&21  \cr
\hline
$\flatrab 4 0 4(n)$ & 134 & 268& 496 & 992& 1856& 3712& 7004& 14014& 26598\cr
\hline
$\fgr 4(n)          $ & 231 & 462& 858 & 1716& 3217& 6435& 12155& 24310& 46189\cr
\hline
\end{tabular} 
\label{eq:szbklngy}
\end{equation}
\begin{equation}
\begin{tabular}{l|r|r|r|r|r|r|r|r|r|r}
$n=$ &   5& 6& 7& 8& 9&10&11&12&13  \cr
\hline
$\flatrab 5 0 5(n)$ & 1 & 1& 2 & 3& 6& 10& 20& 35& 70\cr
\hline
$\fgr 5(n)          $ & 1 & 1& 3 & 5&10& 17& 35& 63& 126\cr
\hline\hline
$n=$ &   14& 15& 16& 17& 18&19&20&21&22  \cr
\hline
$\flatrab 5 0 5(n)$ & 127 & 256& 471 & 942& 1758& 3516& 6620& 13240& 25095\cr
\hline
$\fgr 5(n)          $ & 231 & 462& 858 & 1716&3217& 6435& 12155& 24310& 46189\cr
\hline
\end{tabular} 
\label{eq:szbklot}
\end{equation}
\begin{equation}
\begin{tabular}{l|r|r|r|}
$n$ &298&299&300\cr
\hline
$\flatrab 3 0 3(n)\approx$ & $3.919\,720\cdot 10^{87}$ & $7.839\,440 \cdot 10^{87}$ & $1.562\,662 \cdot 10^{88}$\cr
\hline
$\fghb(n)\approx$& $ 3.932\,918 \cdot 10^{87}$& $7.865\,747 \cdot 10^{87}$& $1.567\,888\cdot 10^{88}$\cr
\hline
$\frac{\fghb(n)}{\flatrab 3 0 3(n)}\approx$& $ 1.003\,367\,003 $& $1.003\,355\,705 $ & $1.003\,344\,482$\cr 
\hline
\end{tabular}
\label{eq:cSmrWlKmlG}
\end{equation}
The computation for the following table took 306 seconds.
\begin{equation}
\begin{tabular}{l|r|r|r|}
$n$ &5\,999&6\,000\cr
\hline
$\flatrab {20} 0 {20}(n)\approx$ & $7.445\,882\,708\,069\cdot 10^{1797}$ & $1.489\,176\,541\,614 \cdot 10^{1798}$ \cr
\hline
$\fgr{20}(n)\approx$& $ 1.488\,924\,847\,889  \cdot 10^{1798}$& $2.977\,849\,695\,779  \cdot 10^{1798}$ \cr
\hline
\end{tabular}
\label{eq:cmjsPjts}
\end{equation}

Next, we give some examples; each of them is based on \eqref{eq:zkRspHjTKc}, Observation \ref{obs:nthNmb}, and one of the computational tables that will be specified.

\begin{example}\label{example:ABCDE}
\textup{(A)} By \eqref{eq:szbklhrm}, $\gmin{\FD 3^{30\,000}}=20$. That is, the direct power  $\FD 3^{30\,000}$ can be generated by 20 elements but not by 19.

\textup{(B)} By \eqref{eq:szbklngy}, $\gmin{\FD 4^{20\,000}}$ is either 20 or 21 but we do not know which one.

\textup{(C)}  By \eqref{eq:szbklot}, $\gmin{\FD 5^{25\,000}}=22$.

\textup{(D)}   By \eqref{eq:cSmrWlKmlG}, $\gmin{\FD 3^{10^{88}}}=300$ (the exponent in the direct power is $10^{88}$).

\textup{(E)}    By \eqref{eq:cmjsPjts}, $\gmin{\FD {20}^{1.489\cdot 10^{1798}}}=6\,000$ (the exponent is $1.489\cdot 10^{1\,798}$).
\end{example}

We know from A000372  of Sloan \cite{sloan} (\href{https://oeis.org/A000372}{https://oeis.org/A000372}) that in spite of lots of work by many contributors, the largest integer $r$ for which $|\FD r|$ is known is $r=9$. 
%Note that apart from its zeroth element, A000372 is the sequence  $2+|\FD{1}|$, $2+|\FD{2}|$ , $2+|\FD{3}|$,\dots{}. It is unlikely that $|\FD{20}|$ gets known in a century. 
We mention the following well-known folkloric lower fact:
\begin{equation}
2^{1024} = 2^{2^{10}} \leq |\FD{20}|.
\label{eq:zrTrnmsLtmg}
\end{equation}
Indeed, the free Boolean lattice $\textup{FB}(10)$ on 10 generators consists of $2^{2^{10}}$ elements and it  is lattice-generated by the free generators of $\textup{FB}(10)$ and their complements. So  $\textup{FB}(10)$ as a distributive lattice is generated  by 20 elements,  implying \eqref{eq:zrTrnmsLtmg}. 

Based on \eqref{eq:zrTrnmsLtmg} and the paragraph above, the direct power in part (E) of Example \ref{example:ABCDE} consists of an unknown but very large number of elements. However, only 306 seconds were needed to determine the least possible size of its generating sets.

\section{Appendix: a Maple worksheet}\label{sect:maple}

We obtained the data in Section \ref{sect:odd} with the use of the computer algebraic program Maple V Release 5 (1997) and a desktop computer (AMD Ryzen 7 2700X Eight-Core Processor 3.70 GHz); 
the whole computation took  7  hours and 16 minutes. The author's Maple program (called a \emph{worksheet} in Maple) is the following.

{\small%\SMALL
\begin{verbatim}
> restart;# A satellite worksheet for  arXiv:2309.13783
> # Decrease the parameters in the #***-marked lines when
> # experiencing with the program or running only a part of the
> # program; otherwise it may run for more than a day.
> alltime0:=time():
> #                             PART A: LOWER BOUNDS
> #                   Part A/1: how to choose p in the theorem
> fpr0r:=proc(p,r,n) local h,s,i,j; h:=p+floor((n-r)/2): s:=0:
>  for i from 0 to floor(n/r)-1 do
>   for j from 0 to i do s:=s+binomial(i,j)*binomial(n-(i+1)*r, h-j*r)
>   od;
>  od;  s:=s;#Since each procedure returns with the last ":="
> end: #end of procedure fpr0r
> findp:=proc(r,n) local p,pmax,s,smax,i; pmax:=-r; smax:=-1;
>  for p from -r to r do s:=fpr0r(p,r,n);
>   if s>smax then smax:=s; pmax:=p
>   fi;
>  od; 
>  if pmax<> 0 then print(`When r=`,r,` n=`,n,` then pmax=`,pmax);
>  fi;            pmax:=pmax;
> end: #end of procedure findp
> # Checking what we guess: p=0 is the best choice
> time0:=time():
> for r from 3 to 200 do #*** 
>  for n from r to 300 do #n from r to 100
>   pfound:=findp(r,n); if pfound<>0 then print(pfound);
>                       fi;
>  od;
> od;          time1:=time():
> print(`Checking that p=0 is the best took `, time1-time0,` seconds.`);
> #    As p=0 proved to be the best, now we can let
> fr0r:=proc(r,n); fpr0r(0,r,n) 
> end: #end of procedure fr0r
> 
> #                             PART B: UPPER BOUNDS
> #                   Part B/1: an easy (but not the best) upper bound
> gr:=proc(r,n) local s; s:=floor(binomial(n+2-r,floor((n+2-r)/2))/2); 
> end:#end of procedure gr
> #         Part B/2/a: where does f33 takes its minimum?
> f33:=proc(n,t,x,y);(t+x)!*(n-t-x)! + (t+y)!*(n-t-y)!
>   + 2*(t+x+y)!*(n-t-x-y)! -2*(t+x)!*y!*(n-t-x-y)!
>   -2*(t+y)!*x!*(n-t-x-y)!
> end: #end of procedure f33   
> time0:=time():
> for n from 3 to 300 do #***
  s0:=f33(n,0,1,1);t0:=floor((n-2)/2);x0:=1;y0:=1;   
  for t from 0 to n-3 do
   for x from 1 to n-t-2 do
    for y from x to n-t-x-1 do # x<=y by symmetry! 
     s:=f33(n,t,x,y); if s<s0 then s0:=s; t0:=t; x0:=x; y0:=y
                      fi
    od
   od
  od; if t0+1<n-3 then if f33(n,t0+1,x0,y0)=s0 then t0:=t0+1 
                       fi 
      fi;
  print(`For n=`,n,` t0=`,t0,` x0=`,x0,` y0=`,y0);#optional
  if x0<>1 or y0<>1 or t0<>floor((n-2)/2) then
   print(`FAILURE, FAILURE, FAILURE, FAILURE for n=`,n);
  fi; 
od:
time1:=time():
> print(`Checking where f33 takes its minimum took `, time1-time0,` seconds.`);
> #         Part B/2/b: computing g3asterix 
> # We compute g3asterix(n) only for those n for which we know that it is
> # the same as g3doubleasterix(n). I.e., only for n<=300.  
> g3asterix:=proc(n) local s0,s,t0,t,x,x0,y,y0,Mn,result;
>  s0:=f33(n,0,1,1);t0:=floor((n-2)/2);x0:=1;y0:=1;   
>  for t from 0 to n-3 do
>   for x from 1 to n-t-2 do
>    for y from x to n-t-x-1 do # x<=y by symmetry! 
>     s:=f33(n,t,x,y); if s<s0 then s0:=s; t0:=t; x0:=x; y0:=y
>                      fi
>    od #for t
>   od #for y
>  od; if t0+1<n-3 then if f33(n,t0+1,x0,y0)=s0 then t0:=t0+1 
>                       fi 
>      fi;
>  if x0<>1 or y0<>1 or t0<>floor((n-2)/2) then
>   print(`FAILURE, FAILURE, FAILURE, FAILURE for n=`,n);
>  else #the conjecture holds for n
>   Mn:=3*floor(n/2)!*ceil(n/2)! + 3*floor((n+2)/2)!*ceil((n-2)/2)! 
>      -6*floor(n/2)!*ceil((n-2)/2)!; result:=floor(n!/Mn);
>  fi;
> end: #End of procedure g3asterix 
> #          Part B/2/c: computing g3doubleasterix(n), 
> # that is, computing g3asterix(n) for n<= 300 FAST.
> g3doubleasterix:=proc(n); 
>  if n>300 then print(`n is too large at present`)
>  else
>   floor(n! / (3*floor(n/2)!*ceil(n/2)!   
>     + 3*floor((n+2)/2)!*ceil((n-2)/2)! -6*floor(n/2)!*ceil((n-2)/2)!))
>  fi;
> end: #end of procedure g3doubleasterix 
> #                             PART C: CHECKING SEPARATION 
> #                   Part C/1: checking that (f303,g3) is separating
> time0:=time(): Print(`Checking that (f303,g3) is separating`);
> for n from 3 to 300 do #***  
>   big:=fr0r(3,n+1):small:=gr(3,n); :quoti:=-1; 
>   if small>0 then quoti:=evalf(big /small);
>   fi:   print(`Hopefully nonnegative=`, big-small, 
>               ` quotient=`,quoti, `  n=`,n ); 
od:
time1:=time(): scn:=time1-time0;
> print(`Checking (f303,g3) is separating took `, scn,` seconds.`);
> #                   Part C/2: checking that (fr0r,gr) is separating
> # for some values of r and n
> time0:=time():
> for r from 3 to 100 do #*** 
>  #print(` `);
>  #print(`Checking that (fr0r,gr) is separating for r=`,r);#optional
>  for n from r to 300 do big:=fr0r(r,n+1):small:=gr(r,n);#*** 
>   quoti:=-1; if small>0 then quoti:=evalf(big /small);
>            fi:  
>   # print(`Hopefully nonnegative=`, big-small, 
>   # ` quotient=`, quoti, `   r=`,r,` n=`,n ); 
>   if big<small then
>    for j from 1 to 100 do 
>      print(`!!! Non-separating for r=`,r,` n=`,n)
>    od; #One can search for "!!! Non-separating" in the output.  
>   fi; 
>  od; #for n
> od:       time1:=time(): scn:=time1-time0:
> print(`Checking (fr0r,gr) is separating took `,scn,` seconds.`);
> 
> #                             Part D: computing some small values
> time0:=time():
> for r from 3 to 1 do #to 22 
>  print(` `); print(`r=`,r);
>  for n from r to 30  do frn:=fr0r(r,n); grn:=gr(r,n);
   if r=3 
   then ggood:=g3doubleasterix(n); print(`r=`,r,
    ` n=`,n,` f303=`,frn,` ggood=`,ggood, ` g3=`,grn);
   else print(`r=`,r,` n=`,n,` fr0r=`,frn,` grn`,grn); 
   fi
 od; #for n
od:  time1:=time(): scn:=time1-time0:
> print(`Computing some small values took `,scn,` seconds.`);
> 
> #                             Part E: computing some large values
> time0:=time(): print(` `); print(`r=`,3);
> for n from 298 to 300 do print(` `); #***  
>  frn:=fr0r(3,n): ggood:=g3doubleasterix(n): grn:=gr(3,n):
>  quoti:=-1; if frn>0 then :quoti:=evalf(ggood/frn):
>           fi: 
>  print(`r=`,3,` n=`,n, ` log[10](frn)=`,evalf(log[10](frn)),
>    ` log[10](ggood)=`,evalf(log[10](ggood)), 
>    ` log[10](grn)=`,evalf(log[10](grn)), ` quotient=`, quoti):
> od;  time1:=time(): scn:=time1-time0:
> print(`Computing some large values took `,scn,` seconds.`):
> 
> #                             Part F: computing some larger values
> time0:=time():r:=20:  print(` `): print(`r=`,r):
> for n from 5999 to 6000 do print(` `);#***   
>  frn:=fr0r(r,n); grn:=gr(r,n); 
>  print(`r=`,r,` n=`,n,` frn=`,frn, ` grn=`, grn,` log[10](frn)=`,
>        evalf(log[10](frn)), ` log[10](grn)=`,evalf(log[10](grn))):
> od:
> time1:=time(): scn:=time1-time0:
> print(`Computing some larger values took `,scn,` seconds.`):
> alltime1:=time(): ascn:=alltime1-alltime0:
> print(`The whole program ran `,ascn,` seconds, i.e. `,
>   evalf(ascn/3600),` hours.`):
\end{verbatim}
}
This worksheet is also available from the \href{http://www.math.u-szeged.hu/~czedli}{author's website}.


\begin{thebibliography}{99}

\bibitem{delbrinczg}
D. Ahmed and G. Cz\'edli: (1+1+2)-generated lattices of quasiorders. Acta Sci. Math.
(Szeged) 87 (2021), 415--427.\footnote{At the time of writing, this paper,  \cite{czglo}, and \cite{zadori2} are also available from  \href{http://www.acta.hu/}{http://www.acta.hu/} as well as many other papers.}



\bibitem{ivanczg}
Chajda, I. and Cz\'edli, G.: How to generate the involution lattice of quasiorders?. Studia Sci.
Math. Hungar. 32 (1996), 415--427.

\bibitem{czgDEBRauth}
Cz\'edli, G.:
Four-generated direct powers of partition lattices and authentication\footnote{{At the time of writing, see the menu item} \href{http://www.math.u-szeged.hu/~czedli/m/listak/publist.html}{``Publications''} {in the author's website for preprints.}}. Publicationes Mathematicae (Debrecen) \tbf{99} (2021), 447--472


\bibitem{czgqoufilt}
G. Cz\'edli: Generating some large filters of quasiorder lattices.
\href{https://arxiv.org/abs/2302.13911}{https://arxiv.org/abs/2302.13911}

\bibitem{czgboolegen}
G. Cz\'edli:  Generating Boolean lattices by few elements and exchanging session keys.
\href{http://arxiv.org/abs/2303.10790}{arXiv:2303.10790}

\bibitem{czgsp}
G. Cz\'edli: 
Sperner theorems for unrelated copies of some partially ordered sets in a powerset lattice and minimum generating sets of powers of distributive lattices.
\href{http://arxiv.org/abs/2308.15625}{arXiv:2308.15625} .


%
\bibitem{czglo}
G. Cz\'edli and L. Oluoch: Four-element generating sets of partition lattices and their direct
products. Acta Sci. Math. (Szeged) 86, 405--448 (2020)


\bibitem{dovegriggs}
Andrew P. Dove,  Jerrold R. Griggs:
Packing posets in the Boolean lattice.
Order \tbf{32},  429--438 (2015)


\bibitem{gelfand}
Gelfand, I.M., Ponomarev, V.A.:
Problems of linear algebra and classification of quadruples of subspaces in a finite dimensional vector space. 
Hilbert Space Operators, Coll. Math. Soc. J. Bolyai 5, Tihany, 1970.
%

\bibitem{GLTFound}
   Gr\"atzer, G.:    Lattice Theory: Foundation.   Birkh\"auser, Basel (2011)


\bibitem{griggsatall}
Griggs, J. R., Stahl, J., Trotter, W. T. Jr.:
A Sperner theorem on unrelated chains of subsets. 
J. Combinatorial Theory, ser. A \tbf{36}, 124--127 (1984)



\bibitem{KatonaNagy}
Katona and Nagy: 
Incomparable copies of a poset in the Boolean lattice.
Order \tbf{32}, 419--427  (2015)

\bibitem{kulin}
J. Kulin: Quasiorder lattices are five-generated. Discuss. Math. Gen. Algebra Appl. 36 (2016),
59--70.


\bibitem{lubell}
Lubell, D: A short proof of Sperner's lemma. 
J. Combinatorial Theory \tbf{1}, 299 (1966)

\bibitem{sloan}
Sloan, N.\ J.\ A.: The On-Line Encyclopedia of Integer Sequence.
\href{https://oeis.org/}{https://oeis.org/}


\bibitem{sperner}
Sperner, E.: Ein Satz \"uber Untermengen einer endlichen Menge. Math. Z. \tbf{27}, 544--548 (1928). \href{https://doi.org/10.1007/BF01171114}{DOI 10.1007/BF01171114}

\bibitem{strietz}
H. Strietz: \"ber Erzeugendenmengen endlicher Partitionverb\"ande. Studia Sci. Math. Hungarica 12 (1977), 1--17. (in German)

\bibitem{takach}
G. Tak\'ach: Three-generated quasiorder lattices. Discuss. Math. Algebra Stochastic Methods
16 (1996) 81--98.

\bibitem{zadori1}
L. Z\'adori: Generation of finite partition lattices. Lectures in universal algebra (Proc. Colloq.
Szeged, 1983), Colloq. Math. Soc. J\'anos Bolyai, Vol. 43, North-Holland, Amsterdam, 1986,
pp. 573--586.

\bibitem{zadori2}
Z\'adori, L.: Subspace lattices of finite vector spaces are 5-generated.
Acta Sci. Math. (Szeged) 74 (2008), 493--499.

\end{thebibliography}
\end{document}